\documentclass{article}
\usepackage[utf8]{inputenc}
\usepackage[dvipsnames]{xcolor}
\usepackage{amsmath}
\usepackage{todonotes}
\usepackage{amsthm}
\usepackage{algorithm2e}
\usepackage{geometry}
\usepackage{soul}
\newcommand{\rbrac}[1]{\left(#1\right)}
\newcommand{\sbrac}[1]{\left[ #1\right]}

\newcommand{\St}[1]{\rbrac{\hspace{-0.05cm}\frac{#1}{e}\hspace{-0.05cm}}^{#1}}
\newcommand{\GG}[2]{g\hspace{-0.08cm}\rbrac{#1}#2\rbrac{\log #1 + 1}}

\newcommand{\type}{\text{type}}

\newcommand{\G}{\mathcal{G}}
\newcommand{\Gnd}{\G_{n,d}}

\newtheorem{theorem}{Theorem}[section]
\newtheorem{prop}[theorem]{Proposition}
\newtheorem{lemma}[theorem]{Lemma}

\title{Zero Forcing Number of Random Regular Graphs}

\author{
Deepak Bal\thanks{Department of Mathematical Sciences, Montclair State University, Montclair, NJ, USA, email: \texttt{deepak.bal@montclair.edu}}
\and
Patrick Bennett\thanks{Department of Mathematics, Western Michigan University, Kalamazoo, MI, USA, e-mail: \texttt{Patrick.Bennett@wmich.edu}}
\and
Sean English\thanks{Department of Mathematics, Ryerson University, Toronto, ON, Canada, e-mail: \texttt{Sean.English@ryerson.ca}}
\and 
Calum MacRury\thanks{Department of Computer Science, University of Toronto, Toronto, ON, Canada, email : \texttt{cmacrury@cs.toronto.edu}}
\and
Pawe\l{} Pra\l{}at\thanks{Department of Mathematics, Ryerson University, Toronto, ON, Canada, e-mail: \texttt{pralat@ryerson.ca}}
}
\begin{document}

\maketitle

\begin{abstract}
The zero forcing process is an iterative graph colouring process in which at each time step a coloured vertex with a single uncoloured neighbour can force this neighbour to become coloured. A zero forcing set of a graph is an initial set of coloured vertices that can eventually force the entire graph to be coloured. The zero forcing number is the size of the smallest zero forcing set. We explore the zero forcing number for random regular graphs, improving on bounds given by Kalinowski, Kam\u{c}ev and Sudakov~\cite{Benny}. We also propose and analyze a degree-greedy algorithm for finding small zero forcing sets using the differential equations method.
\end{abstract}

\section{Introduction}

Let $G=(V,E)$ be any simple, undirected graph. The \textbf{zero forcing process} on a graph $G$ is defined as follows. We start with a subset $S$ of vertices and colour them black; all other vertices are coloured white. At each time step of the process, a black vertex, $v$, with exactly one white neighbour, $w$, will force its white neighbour to become black. In this case we say that \textbf{$v$ forced $w$}. The process stops when no more white vertices can be forced to become black. A \textbf{chronological list of forces} for  a set of black vertices $S$ is a sequence of pairs $(v_1,w_1),\dots,(v_k,w_k)$ such that initially $v_1$ can force $w_1$, and iteratively at the $i$th time step, assuming the that $v_j$ has forced $w_j$ for all $1\leq j< i$, $v_i$ can force $w_i$. It is easy to see that, given an initial set $S$ of black vertices, the set of black vertices obtained by applying the colour-change rule until no more changes are possible is unique, regardless of the order that the forcing is done. Indeed, it follows immediately from the property that once a white vertex is ready to become black it stays at this state until it actually becomes black. Therefore, the following definition is natural. The set $S$ is said to be a \textbf{zero forcing set} if at the end of the zero forcing process all vertices of $V$ become black. The \textbf{zero forcing number} of $G$ is the minimum cardinality of a zero forcing set in $G$, and is denoted by $Z(G)$. This graph parameter is well defined as, trivially, $V$ is a zero forcing set and so $Z(G) \le n$.

The concept of zero forcing was first proposed in~\cite{AIM_2008} as a bound for the minimum rank problem. Given a $n\times n$ symmetric matrix $A=[a_{i,j}]$ over some field $\mathbf{F}$, the graph of $A$, denoted $\mathcal{G}(A)$ is the graph on vertex set $V=\{v_1,\dots,v_n\}$ with edge set $E=\{v_iv_j\mid i\neq j,\,a_{i,j}\neq 0\}$. Note that the graph of a matrix does not depend on the main diagonal of the matrix. The minimum rank of a graph $G$, denoted $\mathrm{mr}^{\mathbf{F}}(G)$ is the minimum rank over all symmetric matrices $A$ over $\mathbf{F}$ such that $\mathcal{G}(A)=G$. The minimum rank problem for graphs has been extensively studied both from the lens of linear algebra and combinatorics. For a survey on the subject, see~\cite{Handbook_2014}. In~\cite{AIM_2008}, it was shown that $Z(G)\geq n-\mathrm{mr}^{\mathbf{F}}(G)$ for any field $\mathbf{F}$. Zero forcing numbers have provided very good bounds for the minimum rank problem, and so have recently been a topic of extensive study. 

In addition to the application to the minimum rank problem, zero forcing sets have been of interest in their own right. Many facets and generalizations of zero forcing have been studied. Some of note are the propagation time, which is defined to be the minimum amount of time for a minimum zero forcing set to force the entire graph, where at each time step all possible forces happen simultaneously~\cite{Hogben_2012}, positive semi-definite zero forcing, a variant which bounds the minimum rank problem when the minimum is taken only over positive semi-definite matrices with graph $G$~\cite{Ekstrand}, and the zero forcing polynomial, which algebraically encodes information about the number of zero forcing sets of a given size in a graph~\cite{Boyer}.

In this paper, we will explore zero forcing for random regular graphs, or more precisely, we consider the probability space of \textbf{random $d$-regular graphs} with uniform probability distribution. This space is denoted $\Gnd$, and asymptotics are for $n\to\infty$ with $d\ge 2$ fixed, and $n$ even if $d$ is odd. We say that an event in a probability space holds \textbf{asymptotically almost surely} (or \textbf{a.a.s.}) if the probability that it holds tends to $1$ as $n$ goes to infinity. Since we aim for results that hold a.a.s., we will always assume that $n$ is large enough. All logarithms which follow, unless otherwise stated, are natural logarithms.

The case of binomial random graphs was studied by Kalinowski, Kam\u{c}ev and Sudakov~\cite{Benny}, who proved that the zero forcing number of $G(n, p)$ is a.a.s. $n - (2 + \sqrt{2} + o(1)) \cdot \frac{\log np}{p}$ for a wide range of $p$. In addition, this group gave bounds involving the spectral properties of regular graphs. These results applied to the random regular graph implies the following theorem.
\begin{theorem}[\cite{Benny}]
Let $d\geq 3$ be fixed. Then a.a.s.
\[
n\left(1-\frac{40\log d}{d}\right)\leq Z(\Gnd)\leq n\left(1-\frac{\log d}{4d}\right).
\]
\end{theorem}
 
Here we improve both the upper and lower bound when $d$ is large.
 
\begin{theorem}
For any fixed $\epsilon>0$, for sufficiently large fixed $d=d(\epsilon)$, we have that a.a.s.
\[
n\left(1-(1+\epsilon)\frac{4\log d}d\right)\leq Z(\Gnd)\leq n\left(1-(1-\epsilon) \frac{\log d}{2d}\right).
\]
\end{theorem}
 
The proof of the upper bound can be found in Section~\ref{section general upper bound}, while the lower bound is proved in Section~\ref{section lower bound}. In addition to these improvements for large $d$, we also give (numerical) upper and lower bounds for some small $d$. In Section~\ref{section numerical upper bounds} using the differential equations method, we propose and analyze a degree-greedy algorithm that finds a zero forcing set. In particular, the analysis of the performance of this algorithm provides the upper bounds presented in Table~\ref{zf-table}. Most of the numerical bounds are calculated in Section \ref{section numerical bounds}, but the best bound for $d=3$ is found in Section \ref{section d=3 improvement}, where we provide a slightly improved algorithm and analysis. The lower bounds are discussed at the end of Section~\ref{section lower bound}.
 
\begin{table}[htbp]
\centering
 \begin{tabular}{c|c|c}
$d$ & lower bound & upper bound \\
\hline 
3 & 0.06992 & 0.17057 \\
4 & 0.14508& 0.25329 \\
5 & 0.21137 & 0.31495 \\
6 &0.26783 & 0.36437 \\
7 &0.31581 & 0.40538 \\
8 &0.35689 & 0.44021 \\
9 &0.39244 & 0.47032 \\
10 &0.42351 & 0.49689 \\
11 &0.45092 & 0.52001 \\
12 &0.47534 & 0.54087 \\
13 &0.49726 & 0.55965 \\
14 &0.51706 & 0.57668 
\end{tabular} 
\caption{Numerical upper and lower bounds on $Z(\Gnd)/n$ that hold a.a.s. for $d=3\ldots 14$ \label{zf-table}}
\end{table} 
 
 
Of particular interest is the bound for cubic graphs, or $3$-regular graphs. Denote the independence number of a graph $G$ by $\alpha(G)$. We say a cubic graph $G$ is \textbf{claw-free} if there is no induced copy of $K_{1,3}$ (the star on 4 vertices) in $G$. Davila and Henning~\cite{Davila} showed that for all connected, claw-free, cubic graphs, we have that $Z(G)\leq \alpha(G)+1$. It has been conjectured that this holds for all cubic graphs. Since all cubic graphs $G$ have $\alpha(G)\geq n/4$, this bound along with our numerical bound for $d=3$ implies that the conjecture is true for almost all cubic graphs.

\subsection{Pairing Model}\label{sec-pairing-model}

Instead of working directly in the uniform probability space of random regular graphs on $n$ vertices $\Gnd$, we use the \textbf{pairing model} or \textbf{configuration model}  of random regular graphs, first introduced by Bollob\'{a}s~\cite{bollobas2}, which is described next. Suppose that $dn$ is even, as in the case of random regular graphs, and consider $dn$ \textbf{configuration points}  partitioned into $n$ labelled buckets $v_1,v_2,\ldots,v_n$ of $d$ points each. A \textbf{pairing} of these points is a perfect matching into $dn/2$ pairs. Given a pairing $P$, we may construct a multigraph $G(P)$, with loops allowed, as follows: the vertices are the buckets $v_1,v_2,\ldots, v_n$, and a pair $\{x,y\}$ in $P$ corresponds to an edge $v_iv_j$ in $G(P)$ if $x$ and $y$ are contained in the buckets $v_i$ and $v_j$, respectively. It is an easy fact that the probability of a random
pairing corresponding to a given simple graph $G$ is independent of the graph, hence the restriction of the probability space of random pairings to simple graphs is precisely $\Gnd$. Moreover, it is well known that a random pairing generates a simple graph with probability asymptotic to $e^{-(d^2-1)/4}$ depending on
$d$. When $d$ is constant, $e^{-(d^2-1)/4}>0$ is as well, so any event holding a.a.s.\ over the probability space of random pairings also holds a.a.s.\ over the corresponding space $\Gnd$. For this reason, asymptotic results over random pairings suffice for our purposes. One of the advantages of using this model is that the pairs may be chosen sequentially so that the next pair is chosen uniformly at random over the remaining (unchosen) points. For more information on this model, see~\cite{NW-survey}.

\subsection{Random 2-regular Graphs}

It can be easily shown that the zero forcing number for random $2$-regular graphs is a.a.s. $(1+o(1))\log n$. We provide the details here for the sake of completeness.

Let $Y_n = Y_n(G)$ be the total number of cycles in any $2$-regular graph $G$ on $n$ vertices (not necessarily random). Since exactly two adjacent vertices are needed to be present in a forcing set, we get that $Z(G) = 2Y_n$.

We know that the random $2$-regular graph is a.a.s.\ disconnected; by simple calculations one can show that the probability of having a Hamiltonian cycle is asymptotic to $\frac 12 e^{3/4} \sqrt{\pi} n^{-1/2}$ (see, for example,~\cite{NW-survey}). We also know that the total number of cycles $Y_n$ is sharply concentrated near $(1/2) \log n$. It is not difficult to see this by generating the random graph sequentially using the pairing model. The probability of forming a cycle in step $i$ is exactly $1/(2n-2i+1)$, so the expected number of cycles is $(1/2) \log n + O(1)$. The variance can be calculated in a similar way. So we get that a.a.s.\ the zero forcing number of a random $2$-regular graph is $(1+o(1)) \log n$.

\section{Numerical Upper Bounds for Small $d$}\label{section numerical upper bounds}

\subsection{Reformulation of the problem}

Let us start with a few observations that will allow us to propose an alternative definition of the zero forcing number and then, based on that, to design a heuristic, greedy algorithm. This algorithm will be used to obtain small zero forcing sets for random $d$-regular graphs and so to get strong upper bounds for $Z(G)$ that holds a.a.s.\ (both numerical and explicit). 

Let $G=(V,E)$ be any graph on $n$ vertices. For convenience, assume that $V = [n] := \{1, 2, \ldots, n\}$. For any $v \in V$, we will use $N(v)$ for the (open) \textbf{neighbourhood} of $v$, that is, the set of neighbours of $v$; $N[v]$ will be used for the \textbf{closed neighbourhood}, that is, $N[v] := N(v) \cup \{v\}$. The \textbf{degree} of $v$ is defined as $\deg(v) := |N(v)|$. Finally, for any $S \subseteq V$, 
$$
N(S) := \left( \bigcup_{v \in S} N(v) \right) \setminus S,
$$
and
$$
N[S] := \bigcup_{v \in S} N[v] = N(S) \cup S.
$$

The main idea behind our reformulation is that the zero forcing set does not need to be fixed in advance. Instead, one can simply start the zero forcing process and try to decide which vertex should be the next one to force, and adjust the zero forcing set accordingly so that this operation is permissible. The idea is quite simple and natural but, unfortunately, the formal definitions are slightly technical. 

In order to build our zero forcing set, we will build a \textbf{Z-sequence}. Given a graph $G$, a sequence $S=(v_1,\dots,v_k)$, where $v_i\in V$ for $1\leq i\leq k$ is a Z-sequence if for each $1\leq i\leq k$, we have
\[
N(v_i)\setminus \bigcup_{j=1}^{i-1} N[v_j]\neq \emptyset.
\]
A maximal Z-sequence is called a \textbf{Z-Grundy dominating sequence}, and the length of a longest Z-Grundy dominating sequence of $G$ is known as the \textbf{Z-Grundy domination number}, denoted $\gamma^Z_{gr}(G)$. The following result of Bre\v{s}ar \textit{et al.}~\cite{Bresar} shows that the zero forcing number of a graph is determined by the Z-Grundy domination number.

\begin{theorem}[\cite{Bresar}]\label{thm:ZGrundy}
	If $G=(V,E)$ is a graph without isolated vertices, then
	\[
	\gamma^Z_{gr}(G)+Z(G)=|V|.
	\]
\end{theorem} 

We provide here an alternative proof of Theorem~\ref{thm:ZGrundy} for the sake of completeness, and also because the proof will be relevant to how our algorithm works. First though we provide a few definitions needed for the proof. Let $S=(v_1,\dots,v_k)$ be a $Z$-sequence. We will say a vertex $w$ is a \textbf{witness} for $v_i$ if $w\in N(v_i)\setminus\bigcup_{j=1}^{i-1}N[v_j]$. Note that for each $Z$-sequence $S=(v_1,\dots,v_k)$, there exists at least one sequence $W=(w_1,\dots,w_k)$ such that $w_i$ is a witness for $v_i$ for all $1\leq i\leq k$. We will call such a sequence $W$ a \textbf{witness sequence for $S$}. In a slight abuse of notation, we will sometimes use $S$ and $W$ to refer to the set of vertices in the sequences $S$ or $W$ respectively. It is worth noting that $W$ and $S$ can have non-empty intersection. We are now ready to present an alternative proof of Theorem~\ref{thm:ZGrundy}.

\begin{proof}[Proof of Theorem \ref{thm:ZGrundy}]
The main idea behind our proof will be to show that the complement of a witness set is a zero forcing set, and vice versa. Since witness sequences and their corresponding Z-sequences are of the same size, this will complete the proof that $\gamma^Z_{gr}(G)+Z(G)=|V|$. In particular, longer Z-Grundy dominating sequences will yield smaller zero forcing sets, and vice versa.
	
Let $S=(v_1,\dots,v_k)$ be a Z-sequence. Let $W=(w_1,\dots,w_k)$ be a witness sequence for $S$. We claim that $B:=V\setminus W$ is a zero forcing set. Indeed, note that initially $v_1$ can force $w_1$ since $N[v_1]\setminus\{w_1\}\subseteq B$ (note that, by definition, for any $2\le i \le k$ we have $w_i \notin N[v_1]$ and so all vertices but $w_1$ are in $B$). Then $v_2$ can force $w_2$, and in general once the vertices $v_1,\dots,v_{i-1}$ have forced $w_1,\dots,w_{i-1}$, $v_i$ can force $w_i$ since $\left(N[v_i]\cup\bigcup_{j=1}^{i-1}N[v_j]\right)\setminus\{w_i\}$ must have all started black or been turned black. Thus, all the vertices in $W$ will become black, so $B$ is a zero forcing set.

Now, let $B$ be a zero forcing set of size $|V|-k$ for some integer $k$.  For some chronological list of forces for $B$, let $v_i$ and $w_i$ denote the $i$th vertex that forced and was forced, respectively, for $1\leq i\leq k$. We claim that $S=(v_1,\dots,v_k)$ is a Z-sequence (not necessarily maximal) with witness sequence $W=(w_1,\dots,w_k)$. Indeed, we have for each $1\leq i\leq k$, $w_i\in N(v_i)\setminus \bigcup_{j=1}^{i-1}N[v_j]$ since after the $(i-1)$st force, $\bigcup_{j=1}^{i-1}N[v_j]$ must all have been turned black, while $w_i$ remains white. Thus, $S$ is a $Z$-sequence with witness sequence $W$. This completes the proof.
\end{proof}

\subsection{Degree-greedy Algorithm}

In light of Theorem~\ref{thm:ZGrundy}, our goal is to build a long Z-Grundy dominating sequence. Indeed, given a graph $G=(V,E)$ on $n$ vertices, if we can build such a sequence of length $\gamma^*(G) \le \gamma^Z_{gr}(G)$, then we have the upper bound on the forcing number, 
$$
Z(G) = n - \gamma^Z_{gr}(G) \le n - \gamma^*(G).
$$

We attempt to build our Z-Grundy dominating sequence greedily. Throughout the algorithm, we maintain a $Z$-sequence $S$ and we also track the set of vertices that $S$ dominates, $T := \bigcup_{v\in S} N[v]$, as well as its complement $U := V\setminus T$. We say that a vertex $v\in T$ is of \textbf{type $r$} if $|N(v)\cap U|=r$.  Note that the sequence $S$ can be extended by any vertex $v$ which has a neighbour in $U$ (that is, any vertex of type $r$ for some $r\ge 1$). Furthermore, if $G$ is connected and $U\neq \emptyset$, then we can always find a vertex from $T$ which extends $S$. We can think of $U$ as a ``reservoir'' of vertices which shrinks by $r$ whenever a vertex from $T$ of type $r$ is added to the Z-sequence. We would like to extend our Z-sequence for as long as possible and so it is natural to always choose a vertex which decreases the size of $U$ by as little as possible. Thus we consider the following greedy algorithm (see Algorithm~\ref{algorithm degree greedy}).

\begin{algorithm}[ht]\label{algorithm degree greedy}
	\SetKwInOut{Input}{Input}
    \SetKwInOut{Output}{Output}
 \Input{Connected graph $G=(V,E)$}
 \Output{Z-Grundy dominating sequence, $S$}
\textbf{Initialization:} \\
Let $v_1$ be a vertex of minimum degree, 
$S = (v_1)$,
$T = N[v_1]$,
$U = V\setminus T$;\\
 \While{$U\neq\emptyset$}{
  Let $v\in T$ be a vertex of minimum positive type\;
 Move vertices of $N(v)\cap U$ from $U$ to $T$\;
 Append $v$ to the end of $S$\;
 }
 \caption{Degree greedy algorithm}
\end{algorithm}
Note that the first step of the while loop is always possible since we assume the input graph is connected. An iteration of the loop is said to be a \textbf{step of type $r$} if the chosen vertex $v$ is of type $r$.

\subsection{Specific Numerical Bounds for $\Gnd$ for Small $d$}\label{section numerical bounds}

In this section, our goal is to analyze how many iterations Algorithm~\ref{algorithm degree greedy} lasts. To this end, we make use of the so-called differential equations method. See~\cite{NW-de} for an extensive survey of the general method.

We will assume throughout this section that $d\geq 3$, that $dn$ is even, and that we are running our algorithm on $G\in \G_{n,d}$. Note that the degree greedy algorithm only produces a Z-Grundy dominating sequence if the graph is connected. Otherwise the algorithm finds a set that is Z-Grundy dominating in just one component. In this case we can say the algorithm fails, but it is very unlikely. Indeed, it was proven independently in~\cite{Bollobas, NW-connectivity} that for constant $d \ge 3$, $G$ is disconnected with probability $o(1)$ (this also holds when $d$ is growing with $n$, as shown in~\cite{Luczak-connectivity}). However, let us stress the fact that we do \emph{not} condition on the fact that $G$ is connected; we simply work with the (unconditional) pairing model allowing the algorithm to finish prematurely if $G$ is disconnected.

Let $T=T(t)$ and $U=U(t)=V \setminus T(t)$ be the two sets of vertices at time $t$ as defined in our algorithm. In order to analyze the algorithm, we will use the pairing model and only reveal partial information about $\G_{n,d}$ at each step. More precisely, all edges within $T(t)$ will be revealed, but no more. If $v$ of type $r$ is being processed, then we reveal the $r$ neighbours of $v$ in $U(t)$, the edges from $T(t)$ to $N(v) \cap U(t)$, and the edges within $N(v) \cap U(t)$, but keep the edges from $T(t+1) = T(t) \cup N(v)$ to $U(t+1) = U(t) \setminus N(v)$ hidden. 

We need to track how many vertices of each type are in $T(t)$. For any $0 \le i \le d$, let $T_i(t)$ denote the number of vertices of type $i$ in $T(t)$ after $t$ steps of the algorithm.  Observe that by the structure of our algorithm, $T_d(t)=0$ in every step of the process. Let 
\begin{align*}
f_{i,j}((t-1)&/n,T_0(t-1)/n,T_1(t-1)/n,\dots,T_{d-1}(t-1)/n):=\\ &\mathbf{E}[T_i(t)-T_{i}(t-1)\mid G[T(t-1)]\text{ is known and step }t\text{ is of type }j].
\end{align*}
In order to simplify the notation, we will write $f_{i,j}(t-1)$ in place of $f_{i,j}((t-1)/n,T_0(t-1)/n,T_1(t-1)/n,\dots,T_{d-1}(t-1)/n)$. Given a statement $A$, we will denote by $\delta_A$, the Kronecker delta function
\[
\delta_A=\begin{cases}1\text{ if }A\text{ is true,}\\ 0\text{ otherwise.}\end{cases}
\]
Further, let us denote by $U(t)= |U(t)| = n-\sum_{\ell=0}^{d-1} T_\ell(t)$. Note that we do not need to track $U(t)$ since it can be determined based on the other variables we are tracking. Then we have 
\begin{align}
f_{i,j}(t-1)~=~&j\cdot {\binom{d-1}i} \cdot \left( \frac{\sum_{\ell=1}^{d-1} \ell T_\ell (t-1)}{d\cdot U(t-1)} \right)^{d-i-1}\left(1-\frac{\sum_{\ell=1}^{d-1} \ell T_\ell (t-1)}{d\cdot U(t-1)}\right)^i\nonumber\\
&+j \cdot (d-1) \cdot \left(\frac{(i+1)T_{i+1}(t-1)-iT_i(t-1)}{d\cdot U(t-1)}\right)\nonumber\\
&-\delta_{i=j}+\delta_{i=0} + O\left( \frac1n\right).\label{eq:fij}
\end{align}
The explanation of the preceding equality is as follows. The first term is the expected contribution of the $j$ vertices from $U(t-1)$ that are adjacent to the vertex $v$ of type $j$ chosen in the algorithm. Indeed, once we expose one neighbour of $v$ in $U(t-1)$, it becomes of type $i$ if precisely $d-1-i$ out of $d-1$ unmatched points associated with this neighbour are matched back to $T(t-1)$. Since all the edges inside $T(t-1)$ have already been revealed, we get that there are exactly $\sum_{\ell=1}^{d-1}\ell T_\ell (t-1)$ configuration points in $U(t-1)$ matched to configuration points in $T(t-1)$. Thus, the expression $(\sum_{\ell=1}^{d-1} \ell T_\ell (t-1))/(d \cdot U(t-1))$ is the probability that a single configuration point in $U(t-1)$ is matched to a configuration point in $T(t-1)$. Of course, in any particular step we may expose more than one (but still at most $O(1)$) edge, so for the entire step $t$ the probability that a point in $U(t-1)$ is matched to a point in $T(t-1)$ is 
\begin{equation}\label{eqn:line2}
    \frac{\sum_{\ell=1}^{d-1} \ell T_\ell (t-1) + O(1)}{d \cdot U(t-1) +O(1)} = \frac{\sum_{\ell=1}^{d-1} \ell T_\ell (t-1) }{d \cdot U(t-1)} +O\rbrac{\frac 1n},
\end{equation}
where the error term is $O(1/n)$ because the numerator $\sum_{\ell=1}^{d-1} \ell T_\ell (t-1)$ is $O(n)$ and (we will assume that) the denominator $d \cdot U(t-1)$ is $\Omega(n)$. The second term in \eqref{eq:fij} accounts for the possibility that vertices that were type $i+1$ could become type $i$ if they end up being adjacent to the neighbours of $v$, the vertex being processed, and also the possibility that a vertex of type $i$ could become type $i-1$ (and here we get another $O(1/n)$ error similar to line \eqref{eqn:line2}). Finally, the Kronecker deltas accounts for the fact that when we process a vertex of type $j$, that vertex becomes type $0$. 

Now, we will consider the algorithm to run in $d-1$ phases, phase $1$ through phase $d-1$. During the first phase, we expect $T    (t)$ to consist mainly of vertices of type $0$ and type $d-1$ since vertices of type less than $d-1$ will be processed at a faster rate than they are produced (on average). Indeed, the algorithm processes vertices of type $d-1$ until some vertex of some type $i \le d-2$ is produced. After processing the first vertex of a lesser type, we typically return to processing vertices of type $d-1$, but after some more steps of this type we may produce another vertex of a lesser type. When vertices of type $d-1$ become plentiful, vertices of lesser type are more commonly created and these hiccups occur more often. When vertices of type $d-2$ take over the role of vertices of type $d-1$, we say (informally!) that the first phase ends and we begin the second phase. In general, in the $k$th phase, a mixture of vertices of type at most $d-k$ are processed. More specifically, in the $k$th phase we typically have a linear number of vertices of type $d-k$, but the algorithm keeps the number of vertices of type less than $d-k$ sublinear.

During the $k$th phase there are, in theory, two possible endings. It can happen that the number of vertices of type $d-k-1$ starts to grow, so we stop processing vertices of type $d-k$ (in which case we move to the next phase). It is also possible that vertices type $d-k$ are getting so rare that those of type $d-k$ disappear (in which case the process goes `backwards') and we begin processing vertices of type $d-k+1$ again. With various initial conditions, either one could occur. However, the numerical solutions of the differential equations for small values of $d$ support the hypothesis that the degree-greedy process we study never goes `back'. The details of the following differential equations method have been omitted, but can be found in~\cite{NW-greedy}.

According to Theorems~1 and~2 from~\cite{NW-greedy}, we can approximate our prioritized algorithm with a deprioritized version that still performs different types of steps in the same proportions as the prioritized algorithm. It means that the vertices are chosen to process in a slightly different way, not always the minimum positive type, but a random mixture of various types. Once a vertex is chosen, it is treated the same way as in the degree-greedy algorithm.

Let us now fix a phase, say phase $k$, and consider the following system of equations with variables $\tau_{i,k}(x,\tilde{y}_0,\tilde{y}_1,\dots,\tilde{y}_d)=\tau_{i,k}(x,\mathbf{\tilde{y}})$, $1\leq i\leq d-k$:

\begin{eqnarray}
	1 &=& \sum_{i=1}^{d-k}\tau_{i,k}(x,\mathbf{\tilde{y}}),\label{eqn:tausys1}\\
	0 &=& \sum_{j=1}^{d-k}\tau_{j,k}(x,\mathbf{\tilde{y}})\cdot f_{i,j}(x,\mathbf{\tilde{y}}), \ \ \ \ \  \text{ for }1\leq i\leq d-k-1.\label{eqn:tausys2}	
\end{eqnarray}

Then if we let $x=t/n$ and $\tilde{y}_i(x)=T_i(t)/n$, 
 the solution $\tau_{i,k}$ of the preceding system of equations can be interpreted as the proportion of steps of type $i$ that occur in phase $k$. Indeed, in phase $k$, vertices of type less than or equal to $d-k-1$ do not accumulate, so they must be processed at the same rate in which they are created, which implies $\sum_{j=1}^{d}\tau_{j,k}(x,\mathbf{\tilde{y}})\cdot f_{i,j}(x,\mathbf{\tilde{y}}) = 0$ for $1\leq i\leq d-k-1$, and since no steps of type greater than $d-k$ occur in phase $k$, we must have $\sum_{i=1}^{d-k} \tau_{i,k}(x,\mathbf{\tilde{y}}) = 1$.

Using the standard differential equation method, this suggests that $\tilde{y}_i$ should be approximately $y_i$ where the $y_i$ are deterministic functions satisfying the following system of differential equations:
\begin{equation}\label{eqn:ysys}
   \frac{dy_i}{dx}=F(x,\mathbf{y},i,k) := \sum_{j=1}^{d-k} \tau_{j,k}(x,\mathbf{y})\cdot f_{i,j}(x,\mathbf{y}). 
\end{equation}

Now, we can say that phase $k$ is the interval $[x_{k-1},x_k]$, where $x_0=0$ and for all $k\ge 1$, $x_k$ is defined to be the infimum of all $x> x_{k-1}$ such that $\tau_{d-k,k}(x,\mathbf{y})=0$. Indeed, this indicates that vertices of type $d-k-1$ begin to build up and do not decrease under repeated processing of vertices of type less than $d-k$, and we move to the next phase. We can then use the final values, $y_i(x_k)$ of phase $k$ as the initial conditions for phase $k+1$. Since we initialize the algorithm with a single vertex of type $d$, we will have our first initial conditions: $y_i(0)=0$ for all $1\leq i\leq d$. The conclusion is that a.a.s., for any $1 \le k \le d-1$, any $0 \le \ell \le d$, and any $x_{k-1} n \le t \le x_k n$, we have
$$
T_\ell (t) = n y_\ell (t/n) + o(n)
$$

Numerical upper bounds are presented in Table~\ref{zf-table}. Below we present some detailed discussion for $d=3$ and $d=4$. 

\subsubsection{$d=3$} Let us concentrate on random 3-regular graphs; that is, on $d=3$. During the first phase, vertices of types 1 and 2 are processed. The proportions of each step type that occur in this phase is shown in Figure~\ref{fig:3}(b). The solution to the relevant differential equations is shown in Figure~\ref{fig:3}(a), namely $y_0(x)$ and $y_2(x)$ (recall that the number of vertices of type 1 is sublinear). The phase ends at time $t_1 \sim x_1 n$ with $x_1 \approx 0.47574$. At this point of the process the number of vertices of type 0 is clearly at least $t_1$ but, in fact, $y_0(x_1) \approx 0.49112$; moreover, $y_2(x_1) \approx 0.15533$. 

\begin{figure}[ht] 
\begin{centering}
\begin{tabular}{ccc}
  \includegraphics[width=0.3\textwidth]{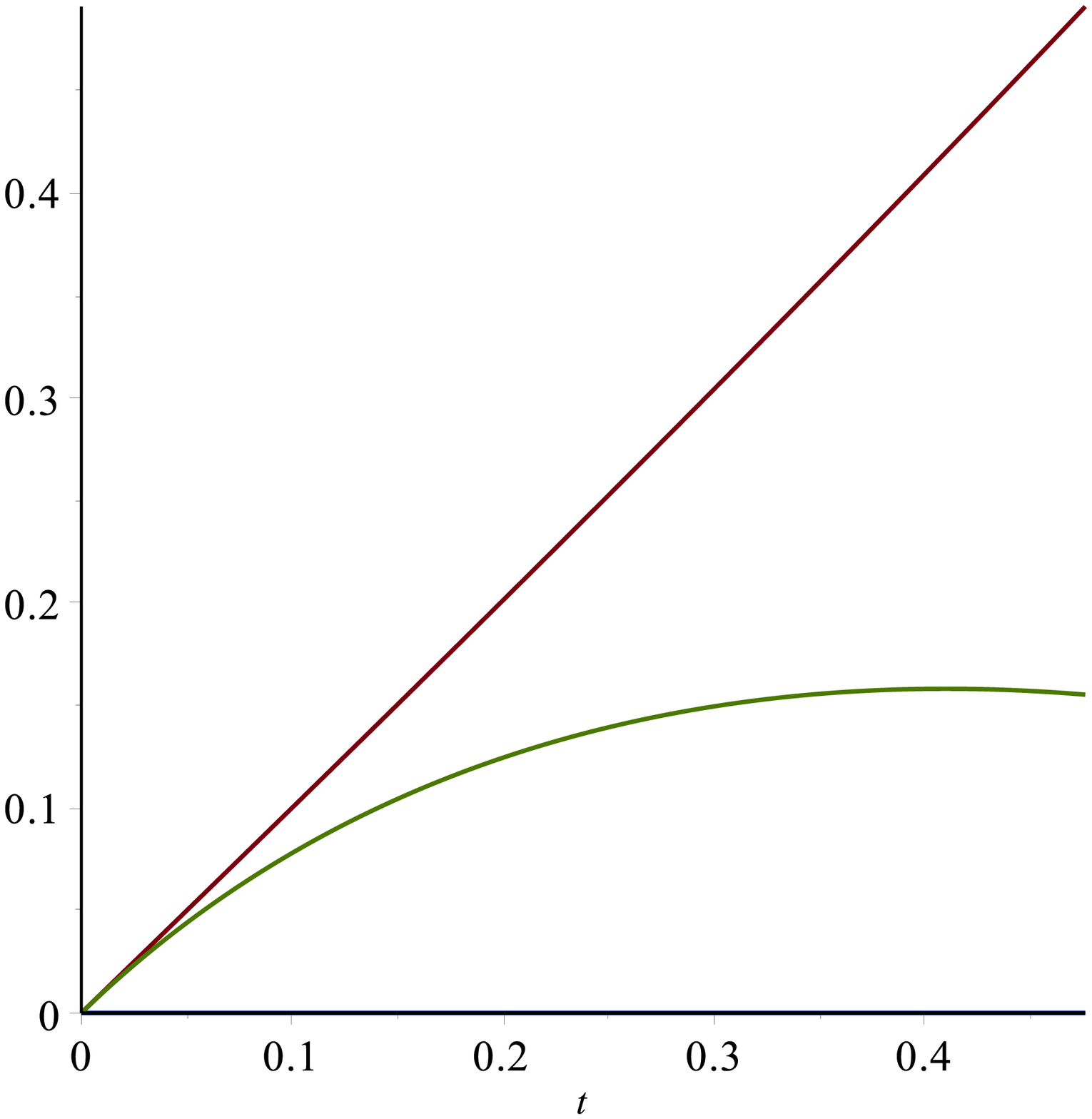} &
  \includegraphics[width=0.3\textwidth]{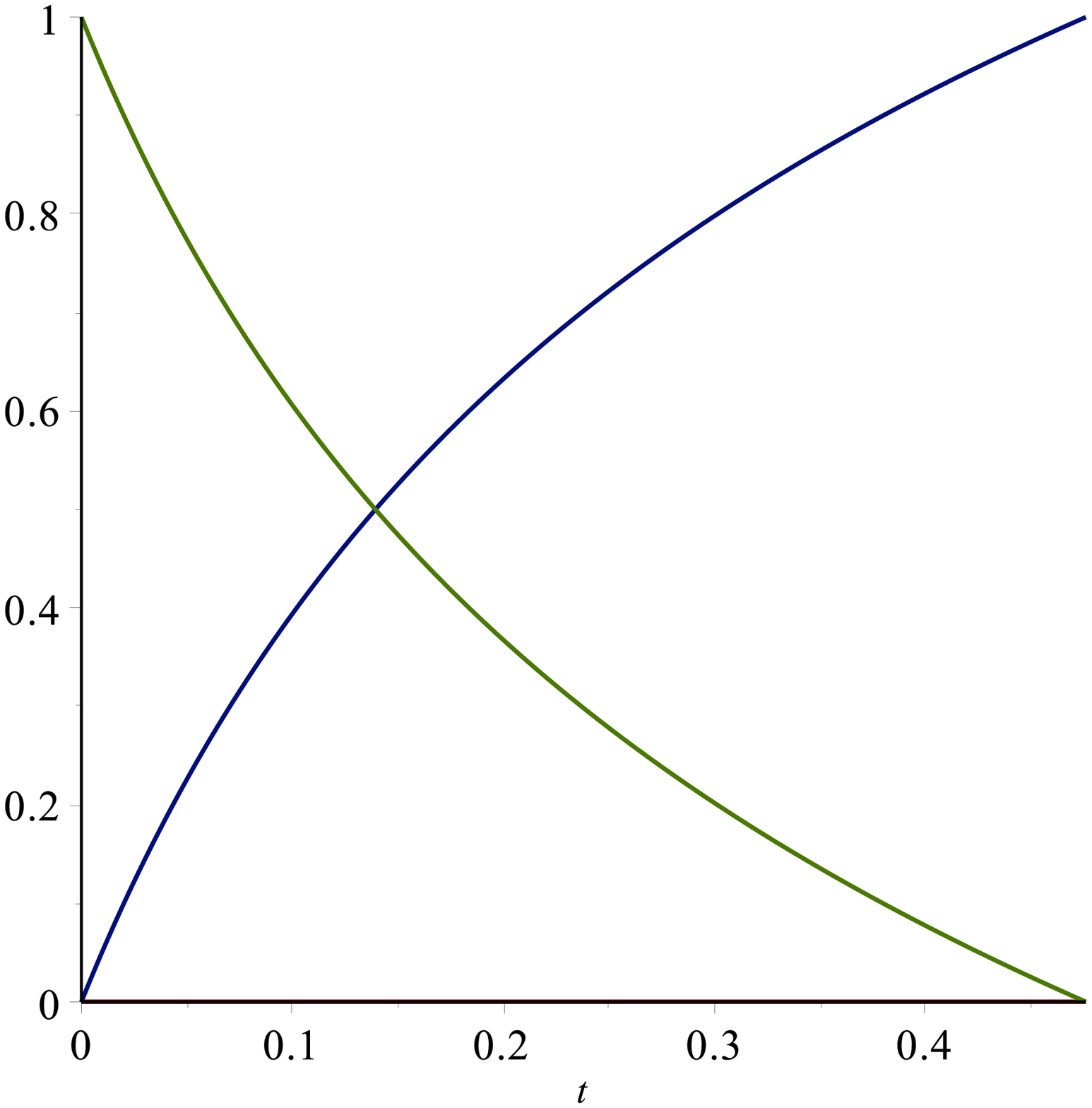} &
  \includegraphics[width=0.3\textwidth]{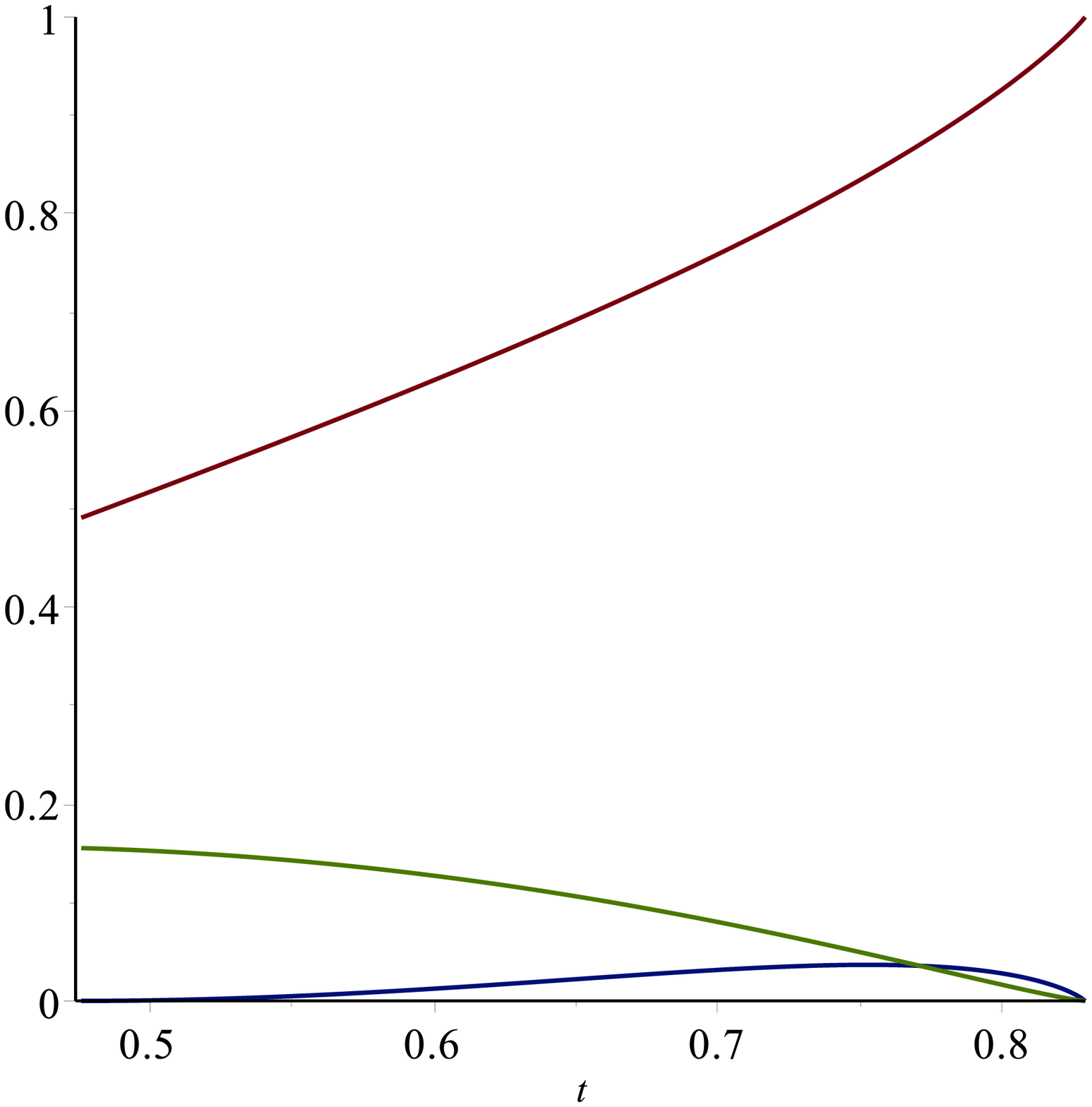} \\
    (a) $y_0$ and $y_2$ &
    (b) $\tau_{1,1}$ and $\tau_{2,1}$ &
    (c) $y_0$, $y_1$, and $y_2$ \\
\end{tabular}\\
\end{centering}
\caption{Solution to the differential equations for $d=3$; phase 1 (a-b) and phase 2 (c).\label{fig:3}}
\end{figure}

During the second phase, only vertices of type 1 are processed. The solution to the relevant differential equations is shown in Figure~\ref{fig:3}(c) The phase ends at time $t_2 = x_2 n$ with $x_2 \approx 0.82929$. In fact, numerical solution implies that $x_2 > 0.82928$ and so we get that a.a.s.\ 
$$ 
Z(G) \le (1-x_2) n \le 0.17072 n \ \ \ \ \ \text{ for } G \in \G_{n,3}.
$$ 

\subsubsection{$d=4$} There are three phases for random 4-regular graphs. During the first phase vertices of types between 1 and 3 are processed; vertices of types 0 and 3 are dominant (vertices of types 1 and 2 are present but the number of them is sublinear)---see Figure~\ref{fig:4}(a). This phase ends at time $t_1 \sim x_1 n$ with $x_1 \approx 0.07167$; $y_0(x_1) \approx 0.07170$ and $y_3(x_1) \approx 0.09858$.

\begin{figure}[ht] 
\begin{centering}
\begin{tabular}{ccc}
  \includegraphics[width=0.3\textwidth]{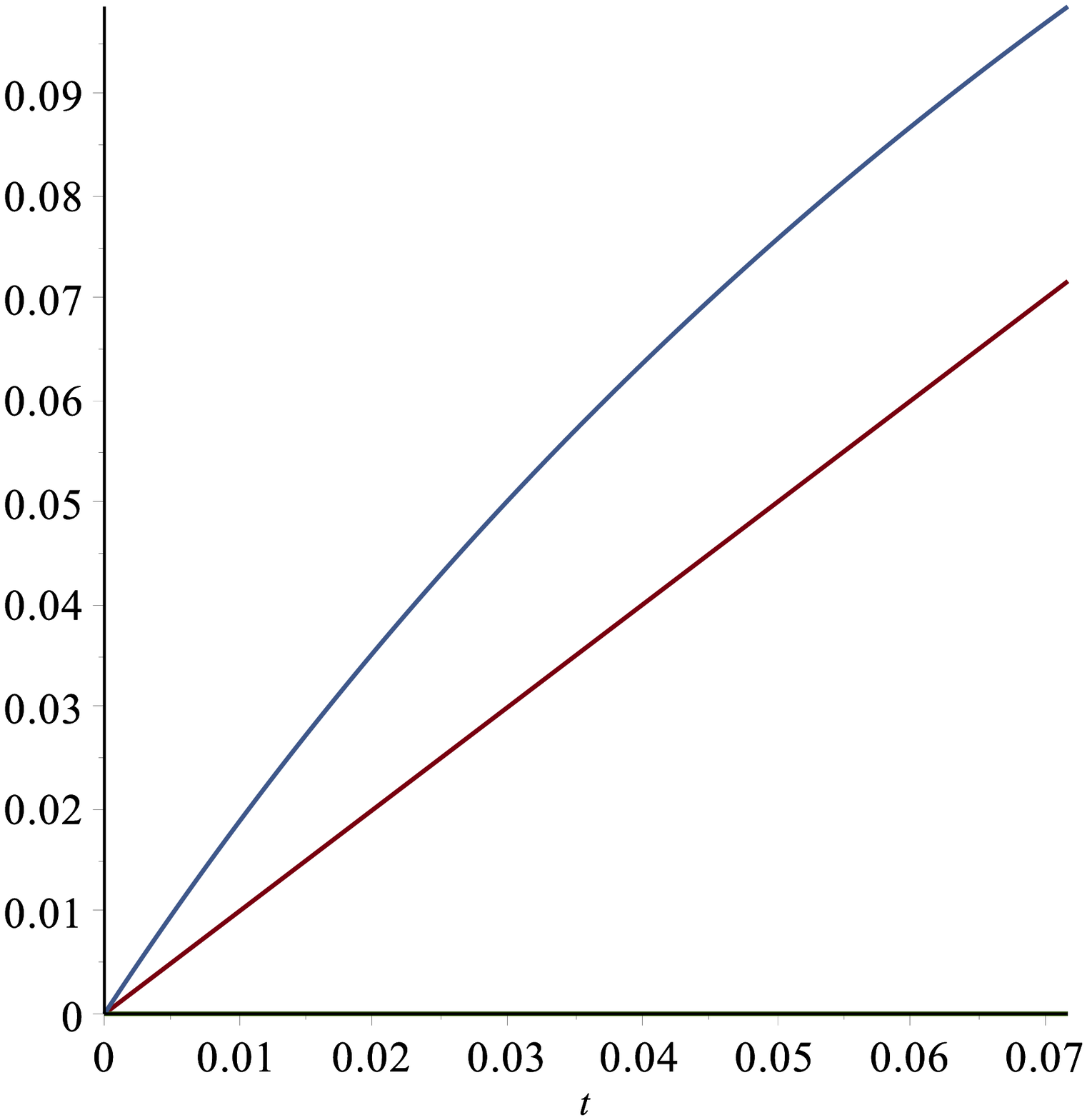} &
  \includegraphics[width=0.3\textwidth]{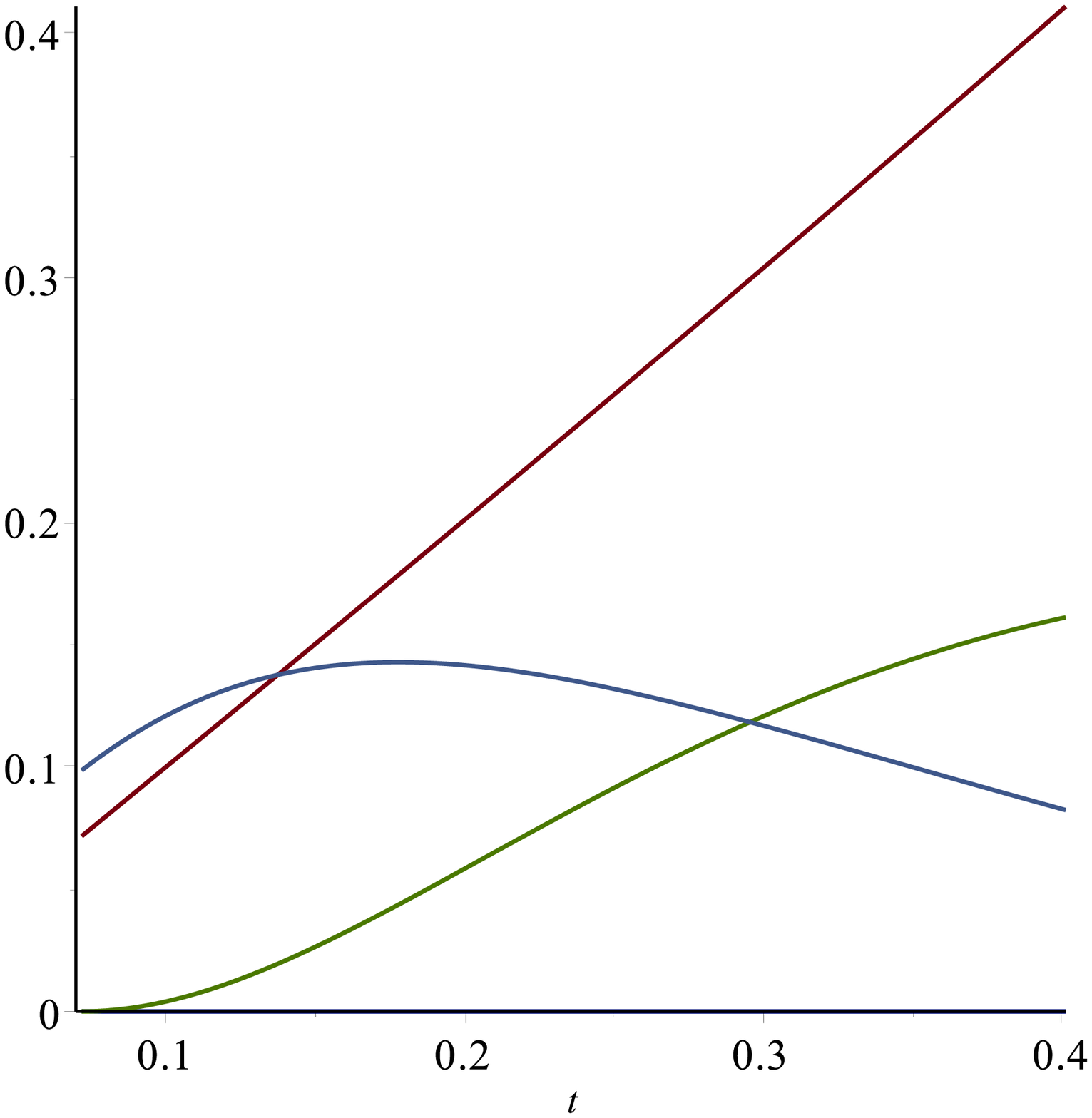} &
  \includegraphics[width=0.3\textwidth]{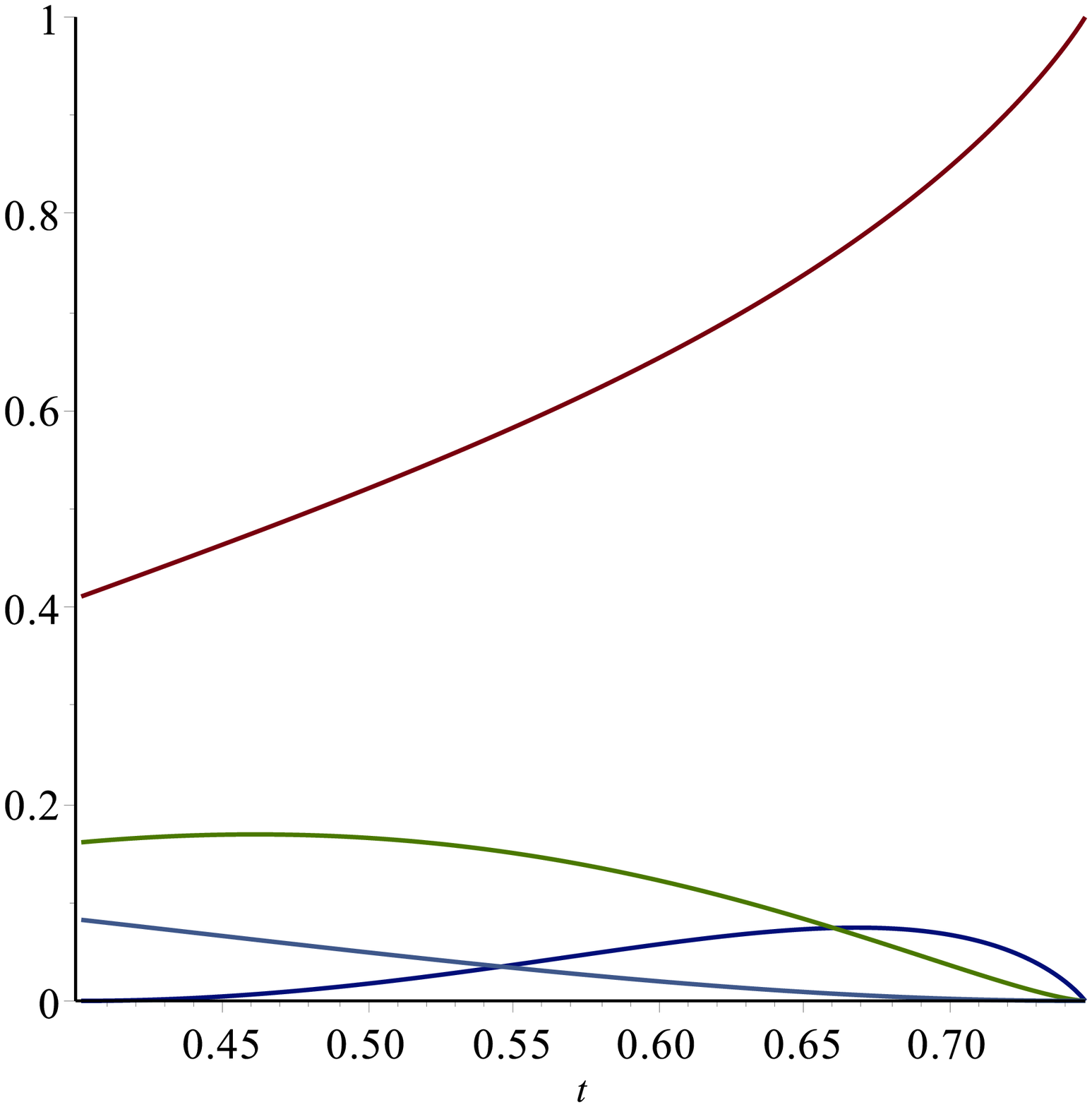} \\
    (a) $y_0$ and $y_3$ &
    (b) $y_0$, $y_2$, and $y_3$ &
    (c) $y_0$, $y_1$, $y_2$, and $y_3$ \\
\end{tabular}\\
\end{centering}
\caption{Solution to the differential equations for $d=4$; phase 1 (a), phase 2 (b), and phase 3 (c).\label{fig:4} }
\end{figure}

During the second phase, vertices of types 1 and 2 are processed; the number of vertices of type 2 becomes linear---see Figure~\ref{fig:4}(b). This phase ends at time $t_2 \sim x_2 n$ with $x_2 \approx 0.40140$; $y_0(x_2) \approx 0.41108$, $y_2(x_2) \approx 0.16120$ and $y_3(x_2) \approx 0.08239$.

The last phase, phase three, processes vertices of type 1, and ends at time $t_3 = x_3 n$ with $x_3 \approx 0.74672$---see Figure~\ref{fig:4}(c). It follows that a.a.s.
$$ 
Z(G) \le (1-x_3) n \le 0.25329 n \ \ \ \ \ \text{ for } G \in \G_{n,4}.
$$

\subsection{Improvement for $d=3$}\label{section d=3 improvement}
In general, we do not expect Algorithm \ref{algorithm degree greedy} to find the longest possible Z-Grundy dominating sequence in a given connected graph. Here we present a modification to the algorithm that does slightly better. The improvement is negligible for large degree $d$, but is noticeable for small $d$, so we only show the results of the improved algorithm for $d=3$.

Let $S$ be a Z-Grundy dominating sequence with witness sequence $W$. The idea behind the improvement is that if there is a vertex $u\in V(G)\setminus S$ such that $N[u]\cap W=\emptyset$, then $u$ can be inserted into $S$, making a longer sequence. In the following algorithm, $S$ will be the Z-sequence we build, $T^{(2)}$ will be a set of witnesses for the Z-sequence $S$, and $T^{(1)}=\left(\bigcup_{v\in S }N[v]\right)\setminus T^{(2)}$. At the end of the process, $T^{(1)}$ will be our zero forcing set. and we will have $T^{(2)}=V\setminus T^{(1)}$. The set $U$ will play essentially the same role it did in Algorithm \ref{algorithm degree greedy}, while the set $T$ from Algorithm \ref{algorithm degree greedy} is split in two here, sets $T^{(1)}$ and $T^{(2)}$.

\begin{algorithm}[ht]\label{algorithm black claws}
	\SetKwInOut{Input}{Input}
	\SetKwInOut{Output}{Output}
	\Input{Connected graph $G$}
	\Output{Z-Grundy dominating sequence, $S$}
	Initialization:\\
	Let $v_1$ be a vertex of minimum degree. 
	$S = (v_1)$,
	$T^{(1)} = \{v_1 \}$,
	$T^{(2)} = \{\}$,
	$U = V\setminus \{v_1\}$\\
	\While{$U\neq\emptyset$}{
		Let $v\in T^{(1)}\cup T^{(2)}$ be a vertex of minimum positive type\;
		\If{$v$ is an element of $T^{(1)}$ of type $\geq 2$ and there exists $u\in N(v)\cap U$ such that $N(u)\subseteq T^{(1)}$}{
			Append $u$ to the end of $S$\;
			Move $v$ from $T^{(1)}$ to $T^{(2)}$\;
            Move $u$ from $U$ to $T^{(1)}$\;
		}
		Arbitrarily choose a vertex $w\in N(v)\cap U$ that minimizes the quantity $|(N(w)\setminus N(v))\cap U|$\;
		Append $v$ to the end of $S$\;
		Move $w$ from $U$ to $T^{(2)}$\;
		Move the vertices of $(N(v)\setminus\{w\})\cap U$ from $U$ to $T^{(1)}$\;
	}
	\caption{Smart degree greedy algorithm}
\end{algorithm}

The \textbf{if} checks to see if we can insert a vertex $u$ into the sequence, giving us a longer sequence. Furthermore, this algorithm gives preference to vertices that will end up being a lower type when choosing which vertex to add to the witness set $W=T^{(2)}$. The rationale behind this is that we only get the extra savings from this new algorithm when we process a vertex of type $\geq 2$ that is not in $W$, so we prioritize adding vertices of higher type to $T$ to make this happen more often.

To analyze this algorithm, we will assume that $n$ is even and we are running it on $G=\mathcal{G}_{n,3}$. we need to track a few more variables than Algorithm \ref{algorithm degree greedy}. Here, let $T^{(1)}(t)$, $T^{(2)}(t)$ and $U(t)$ be the vertices in the sets $T^{(1)}$, $T^{(2)}$, and $U$ at iteration $t$ of the algorithm. Similarly to before, we will say a vertex in $T^{(1)}(t)\cup T^{(2)}(t)$ is of type $r$ if it has $r$ neighbours in $U(t)$. For each $0\leq i\leq 2$, and $k\in \{1,2\}$, let $T_{i,k}(t)$ be the number of vertices of type $i$ in $T^{(k)}(t)$. For ease of notation, we will let $T_{i}(t)=T_{i,1}(t)+T_{i,2}(t)$ and let $U(t)=n-T_0(t)-T_1(t)-T_2(t)$.

We can analyze this algorithm using the differential equations method similarly to our analysis in Section \ref{section numerical bounds}. We will not provide all the details as it is a routine extension of our work in Section \ref{section numerical bounds}, instead we will simply provide the expected changes in our tracked variables and then provide the numerical solution for this case. Before we needed only to track the change in vertices of type $i$ when processing a vertex of type $j$, but here we also need to distinguish between vertices in $T^{(1)}$ and $T^{(2)}$. Thus, we will let $f_{i,j,k,\ell}(t-1)$ track the expected change in $T_{i,k}(t-1)$, given that we process a vertex of type $j$ in $T^{(\ell)}(t-1)$, and given that we have revealed all edges within $T^{(1)}(t-1)\cup T^{(2)}(t-1)$. For the sake of brevity, we will suppress all instances of $(t-1)$ in our variables $T_i(t-1)$, $T_{i,j}(t-1)$ and $U(t-1)$. Let $P := \frac{T_1+2T_2}{3\cdot U}$, so $P$ represents the probability that an exposed edge lands in $T^{(1)}\cup T^{(2)}$. For $a\geq b$, let $P_{a,b}:=2^{\delta_{a\neq b}}\binom2a \binom2b P^{4-a-b}(1-P)^{a+b}$. Then $P_{a,b}$ is the probability that when we process a vertex of type $2$, the neighbours become types $a$ and $b$. Then

\begin{align}
&f_{i,j,k,\ell}(t-1)~=\nonumber\\[0.15cm]
&2j \cdot \left(\frac{(i+1)T_{i+1,k}-iT_{i,k}}{3\cdot U}\right)+\delta_{j=1}\rbrac{\delta_{i=0,k=\ell}-\delta_{i=1,k=\ell}+\delta_{k=2}\binom{2}{i}P^{2-i}(1-P)^i}\label{equation d=3 line 2}\\
&+\delta_{j=2}\left[\sum_{2\geq a\geq b\geq 1} P_{a,b}(\delta_{i=a,k=1}+\delta_{i=b,k=2}+\delta_{i=0,k=\ell}-\delta_{i=2,k=\ell})\right.\label{equation d=3 line 3}\\
&+\sum_{1\le a\le 2}P_{a,0}\left(\delta_{\ell=1}\left(\frac{T_{1,1}+2T_{2,1}}{T_1+2T_2}\right)^2(\delta_{i=a,k=2}+\delta_{i=0,k=1}+\delta_{i=0,k=2}-\delta_{i=2,k=1})\right.\label{equation d=3 line 4}\\
&+\left.\left(1-\delta_{\ell=1}\left(\frac{T_{1,1}+2T_{2,1}}{T_1+2T_2}\right)^2\right)(\delta_{i=a,k=1}+\delta_{i=0,k=2}+\delta_{i=0,k=\ell}-\delta_{i=2,k=\ell})\right)\label{equation d=3 line 5}\\
&+P_{0,0}\left(\delta_{\ell=1}\left(1-\left(1-\left(\frac{T_{1,1}+2T_{2,1}}{T_1+2T_2}\right)^2\right)^2\right)(\delta_{i=0}+\delta_{i=0,k=2}-\delta_{i=2,k=1})\right.\label{equation d=3 line 8}\\
&+\left.\left.\left(1-\delta_{\ell=1}\left(1-\left(1-\left(\frac{T_{1,1}+2T_{2,1}}{T_1+2T_2}\right)^2\right)^2\right)\right)(\delta_{i=0}+\delta_{i=0,k=\ell}-\delta_{i=2,k=\ell})\right)\right]\label{equation d=3 line 9}
\end{align}

The explanation for the equation above is as follows. There are two main changes that we need to account for in this algorithm that were not in Algorithm \ref{algorithm degree greedy}, namely we need to account for when the \textbf{if} clause is triggered, and in the algorithm right after the \textbf{if} clause ends, we need to make sure the vertex we chose minimizes the quantity $|(N(w)\setminus N(v))\cap U|$. 

In line \eqref{equation d=3 line 2}, the term $2j \cdot \left(\frac{(i+1)T_{i+1,k}-iT_{i,k}}{3\cdot U}\right)$ accounts for the situation when a vertex in $T^{(1)}\cup T^{(2)}$ ends in being neighbours with the neighbours of $v$. The other term in line \eqref{equation d=3 line 2} accounts for when the vertex $v$ being processed is of type $1$. More precisely, in this case, $v$ changes from type $1$ to type $0$. Furthermore, the neighbour of $v$ in $U$ moves from $U$ to $T^{(2)}$, and is of type $i$ with probability $\binom2i P^{2-i}(1-P)^i$. The remaining lines of the equation account for when we process a vertex of type $2$.

In line \eqref{equation d=3 line 3} inside the square brackets, we account for when neither neighbour of the vertex being processed ends up type $0$ (i.e. in these cases, we know the \textbf{if} clause of the algorithm is not triggered). We have probability $P_{a,b}$ of the neighbours being types $a$ and $b$, and then in this case, since $a\geq b$,
we move the neighbor of $v$ of type $a$ to $T^{(1)}$ and the neighbour of type $b$ to $T^{(2)}$, and then $v$ goes from type $2$ to type $0$.

In lines \eqref{equation d=3 line 4} and \eqref{equation d=3 line 5} we deal with the situation where one of the neighbours of $v$ is type $0$ and the is of type $1$ or $2$. If $v$ is in $T^{(1)}$, then we need to consider the possibility that the \textbf{if} clause is triggered. This happens with probability $\left(\frac{T_{1,1}+2T_{2,1}}{T_1+2T_2}\right)^2$ (since we are conditioning on the fact that exactly one neighbour of $v$ is of type $0$), and in this case we move the positive type vertex to $T^{(2)}$, the type $0$ vertex to $T^{(1)}$, and then $v$ becomes type $0$ and moves to $T^{(2)}$. Then in line \eqref{equation d=3 line 5}, we deal with when the \textbf{if} clause was not triggered, in which case we prioritize adding vertices to $T^{(1)}$ of higher type, so the positive type neighbour of $v$ is moved to $T^{(1)}$, the type $0$ neighbour is moved to $T^{(2)}$, and then $v$ becomes type $0$.

Then lines \eqref{equation d=3 line 8} and \eqref{equation d=3 line 9} account for when $v$ is of type $2$ and both neighbours of $v$ become type $0$. In this case, if $v$ is in $T^{(1)}$, then we have probability $1-\left(1-\left(\frac{T_{1,1}+2T_{2,1}}{T_1+2T_2}\right)^2\right)^2$ of one of the neighbours of $v$ having all neighbours in $T^{(1)}$ since we are conditioning on both neighbours of $v$ being type $0$. If this is the case, then the \textbf{if} clause is triggered, so we add one vertex to $T^{(1)}$ and one to $T^{(2)}$ of type $0$ (note  $\delta_{i=0}=\delta_{i=0,k=1}+\delta_{i=0,k=2}$). Then $v$ becomes type $0$ and moves from $T^{(1)}$ to $T^{(2)}$. Finally, on line \eqref{equation d=3 line 9}, if neither neighbour of $v$ has all of its neighbours in $T^{(1)}$, then we add a vertex of type $0$ to each of the sets $T^{(1)}$ and $T^{(2)}$, and then $v$ goes from type $2$ to type $0$. This finished the explanation for the expression for $f_{i,j,k,\ell}$.

We will omit most of the remaining details for the rest of the analysis of Algorithm \ref{algorithm black claws}, since it is very similar to  Section \ref{section numerical bounds}. To summarize the main steps we are omitting, we define $\tau_{i, k, m}$ to be the solution to a system that resembles \eqref{eqn:tausys1}-\eqref{eqn:tausys2}, with the interpretation that $\tau_{i,k,m}$ is the proportion of steps of type $i$ in which we process a vertex from $T^{(k)}$ in phase $m\in \{1,2\}$. It is worth noting that since Algorithm \ref{algorithm black claws} does not prioritize processing vertices in $T^{(1)}$ over $T^{(2)}$, or vice versa in the first step of the \textbf{while} loop, this system has one degree of freedom. For our analysis, in phase $1$, we will assume in the \textbf{while} loop, whenever we are processing a vertex of type $2$, we choose if we process a vertex from $T^{(1)}$ or $T^{(2)}$ randomly with probability proportional to the size of $T^{(1)}$ and $T^{(2)}$. This corresponds to adding in the equation $\tau_{2,1,1}(x,\mathbf{\tilde{y}})\cdot T_{2,2}(x,\mathbf{\tilde{y}})=\tau_{2,2,1}(x,\mathbf{\tilde{y}})\cdot T_{2,1}(x,\mathbf{\tilde{y}})$.

Then we let vector-valued function ${\bf y}$ be the solution to a differential equation similar to \eqref{eqn:ysys}. Analogously to the analysis of Algorithm \ref{algorithm degree greedy} for $d=3$ there are two phases. Numerical solutions to the system we have described give a numerical bound of $0.17057$ (an improvement of $0.00015$ over Algorithm \ref{algorithm degree greedy}).

The way in which we choose between processing vertices in $T^{(1)}$ and $T^{(2)}$ can affect the performance of the algorithm. For example, if we always prioritize processing vertices from $T^{(1)}$ over vertices of $T^{(2)}$, Algorithm \ref{algorithm black claws} actually performs identically to Algorithm \ref{algorithm degree greedy} as in this case the vertices of positive type in $T^{(1)}$ do not accumulate, and so the number of vertices we find that can be inserted into our $Z$-sequence is negligible. It may be possible to improve further on our analysis by changing how the algorithm chooses between sets $T^{(1)}$ and $T^{(2)}$ in the first step of the \textbf{while} loop.

\section{Explicit Upper Bounds for large $d$}\label{section general upper bound}

In this section, we provide explicit bounds on the forcing number of $d$-regular graphs.
These bounds are derived from tools based in spectral graph theory; in particular, the expander mixing lemma is used frequently. In what follows, we assume that $G$ is a $n$-vertex $d$-regular graph,
with eigenvalues of its adjacency matrix listed in decreasing order as $\lambda_{1}(G) \ge \ldots \ge \lambda_{n}(G)$.
We refer to $G$ as a $(n,d,\lambda)$ graph, provided $|\lambda_{i}| \le \lambda$ for each $i=2, \ldots ,n$,
where $\lambda \ge 0$. 

\medskip
Let us start with two well-known lemmas.

\begin{lemma}\label{lem:eigenvalue_properties}
If $G=(V,E)$ is a $(n,d,\lambda)$ graph, then the following statements hold:
\begin{enumerate}
\item $\lambda_{1}(G)=d$. 
\item If $\lambda < d$, then $G$ is connected.
\item $\lambda \ge \sqrt{d} - \frac{d}{\sqrt{n}}$.
\end{enumerate} 
\end{lemma}

\begin{lemma}[Expander Mixing Lemma]\label{lem:expander_mixing}
If $G=(V,E)$ is a $(n,d,\lambda)$ graph, and $U,W \subseteq V(G)$, then 
\[
\left| \frac {d|U||W|}{n} - e(U,W) \right| \le \lambda \sqrt{|U||W| 
\left( 1- \frac {|U|}{n} \right) \left( 1- \frac {|W|}{n} \right)},
\]
where $e(U,W):=|\{(u,w) \in U \times W: \{u,w\} \in E\}|$ (edges with both ends in $U \cap W$
are counted twice).
\end{lemma}

\bigskip

We first observe that Lemma \ref{lem:expander_mixing} guarantees
the existence of edges between subsets of sufficiently large size.

\begin{prop}\label{prop:spectral_edge_guarantee}
For any $U,W \subseteq V(G)$, if $u:= \min\{ |U|, |W| \} > \frac{\lambda}{d+\lambda} \; n,$
then $e(U,W)>0$.
\end{prop}

\begin{proof}
Indeed, it follows immediately from the Expander Mixing Lemma that
\begin{eqnarray*}
e(U,W) &\ge& \frac {d|U||W|}{n} - \lambda \sqrt{|U||W| 
\left( 1- \frac {|U|}{n} \right) \left( 1- \frac {|W|}{n} \right)} \\
&=& \frac {|U||W|}{n}\left[ d  - \lambda \sqrt{ 
\left(  \frac {n}{|U|} -1 \right) \left(  \frac {n}{|W|} -1 \right)} \right] \\
&\ge& \frac{u^2}{n} \left[d - \lambda \left(\frac{n}{u} - 1 \right) \right] > 0.
\end{eqnarray*}
The result holds. 
\end{proof}

\bigskip

We are now ready to analyze Algorithm \ref{algorithm degree greedy} for
general $d$. This will provide us with upper bounds on the forcing number 
of arbitrary $d$-regular graphs. That being said, it will be convenient to assume
that $G$ is connected.

Recall that Algorithm \ref{algorithm degree greedy} works by building a $Z$-sequence 
$S=(v_{1}, \ldots ,v_{k})$ of the input graph $G=(V,E)$. For each $1 < t \le k$,
let us denote by $B_{t}$ the set $\bigcup_{i=1}^{t-1}N[v_i]$, and $W_{t}= V(G)\setminus B_t$. We may then denote $B_{1}:= \emptyset$ and $W_{1}:=V(G)$. We will call the vertices in $B_t$ black and the vertices in $W_t$ white.

For each vertex $v \in V(G)$, we define its \textit{type} at time $1\le t \le k$, denoted $\type_{t}(v)$ (or $\type(v)$ when clear), to be the number of neighbours $v$
has in $W_{t}$, namely $\deg_{W_{t}}(v)$. We may then partition $B_{t}$ into the sets $B_{t}^{0}, \ldots ,B_{t}^{d}$, where for $0 \le i \le d$, $B_{t}^{i}$ is defined to be the
vertices of $B_{t}$ with type $i$. Similarly, we may partition $W_{t}$ into the sets $W_{t}^{0}, \ldots ,W_{t}^{d}$. It will be convenient to
denote $B_{t}^{+}$ and $W_{t}^{+}$ as $\bigcup_{i=1}^{d}B_{t}^{i}$ and $\bigcup_{i=1}^{d}W_{t}^{i}$, respectively.

We initialize Algorithm \ref{algorithm degree greedy} by first selecting the vertex $v_{1}$ arbitrarily. For $t >1$, it then chooses $v_{t}$ 
among the vertices of $B_{t}$ whose type is positive and nonzero; that is, we choose $v_t$ from the members of $B_{t}^{+}$. In particular,
it looks to select a vertex of this kind whose type is minimal. The algorithm finishes executing when no such vertex exists. Observe
that since the graph $G$ is connected, once the vertex $v_{k}$ is processed, all the vertices of the graph must be coloured black.
In particular, we know that $|W_{k}| \le d$.

Let us first consider an essential time in the execution of the algorithm. Namely,
define $t_{1}$ as the smallest $1 < t \le k$, for which
\[
 |W_{t}| \le \frac{\lambda}{d+\lambda}n.
\]
It is clear from the above discussion that $t_{1}$ exists. We may therefore think of phase one as those $t \ge 1$, such that $1 \le t < t_{1}$ (note that the phases referred to in this section are different than the phases referred to in Section \ref{section numerical upper bounds}).
Similarly, we may define the remaining time increments as the second and final phase. It turns out that in every iteration of the algorithm in phase one, we are guaranteed to find
a black vertex of type proportional to $|W_{t}|$.

\begin{lemma}
Let $G=(V,E)$ be a $(n,d,\lambda)$ graph, and assume that it is passed as input to Algorithm~\ref{algorithm degree greedy}. For each $1 < t < t_{1}$,
there exists a vertex $u \in B_{t}$, such that
\[
	1 \le \type_{t}(u) \le (d + \lambda)\frac{|W_{t}|}{n} + \lambda.
\]
\end{lemma}

\begin{proof}
Observe that by Lemma~\ref{lem:expander_mixing}, we are guaranteed a
lower bound on twice the number of edges within $B_{t}^{+}$. In particular,
\[
	e(B_{t}^{+},B_{t}^{+}) \ge |B_{t}^{+}|\left(\frac{d+\lambda}{n}|B_{t}^{+}| - \lambda \right).
\]
As a result, there exists a vertex $u \in B_{t}^{+}$, for which
\[
	\deg_{B_{t}^{+}}(u) \ge \frac{d+\lambda}{n}|B_{t}^{+}| - \lambda,
\]
where $\deg_{B_{t}^+}(u)$ is the number of neighbours $u$ has in $B_{t}^{+}$.
Thus, we have that
\begin{align*}
\type_{t}(u) &= d -\deg_{B_{t}}(u)	\\
		&=d - \deg_{B_{t}^{+}}(u) -\deg_{B_{t}^{0}}(u) 	\\
		&\le d - \frac{d+\lambda}{n}|B_{t}^{+}| + \lambda.
\end{align*}
However, $|B_{t}^{+}|= n - |W_{t}| - |B_{t}^{0}|$, so after simplification,
\[
	\type_{t}(u) \le \frac{d + \lambda}{n}|W_{t}| + \frac{d+ \lambda}{n}|B_{t}^{0}|.
\]

It remains to bound the size of $|B_{t}^{0}|$.
Recall that we are in phase 1, which implies
that $|W_{t}| > \frac{\lambda}{d+ \lambda}n$. On the other hand, we know
that $e(B_{t}^{0},W_{t})=0$ by definition of $B_{t}^{0}$. Thus, by Proposition  \ref{prop:spectral_edge_guarantee},
we have that $|B_{t}^{0}| \le \frac{\lambda}{d+ \lambda}n$. Applying this observation to the above equation,
we get that
\[
	\type_{t}(u) \le \frac{d + \lambda}{n}|W_{t}| + \lambda,
\]
thereby proving the lemma, as $\type_{t}(u) \ge 1$ by definition.
\end{proof}

We can think of this lemma as guaranteeing a worst case selection throughout
the first phase. For each $1 < t < t_{1}$, the algorithm selects a vertex $v_{t} \in B_{t}^{+}$, for which
$\type(v_{t}) \le (d + \lambda)\frac{|W_{t}|}{n} + \lambda$. Since $\type(v_{t})$ many vertices are removed from $W_{t}$
when $v_{t}$ is processed, for each $1 <t < t_{1}$ we have
\begin{equation}\label{equation w_t recursion}
	|W_{t+1}| \ge |W_{t}| - (d + \lambda)\frac{|W_{t}|}{n} - \lambda.
\end{equation}

Let us now define $w_{t}:= |W_{t}|$ for each $t \ge 1$. We shall use the above recursion to lower bound the length of phase one, namely
$t_{1}$. As the length of the Z-sequence we are generating is precisely the number of iterations the algorithm performs, this will give us a lower bound on the length of the $Z$-sequence, and as a consequence of Theorem \ref{thm:ZGrundy}, an upper bound on the forcing number of $G$. 

\begin{prop}\label{prop:general_forcing_number_bound}
Let $G=(V,E)$ be a connected $(n,d,\lambda)$-graph, for which $3\leq d \le \sqrt{n}$.
In this case, we have that
\[
	Z(G) \le n - \log \left( \frac{2\lambda}{d+ \lambda}\right)\left(\log \left({1 - \frac{d+ \lambda}{n}}\right) \right)^{-1}=n-(1+o(1))\log \left( \frac{d+\lambda}{2\lambda}\right)\frac{n}{d+\lambda}.
\]
\end{prop}

\begin{proof}
As above, let us assume that we pass $G$ to Algorithm \ref{algorithm degree greedy}, 
which returns the $Z$-sequence $S=(v_{1},\ldots,v_{k})$.
We denote $t_{1}$ as the first time $1\le t \le k$ for which $w_{t} \le \frac{\lambda}{d+\lambda}n$. Our goal will be to lower bound the size of $t_{1}$.

By inequality~(\ref{equation w_t recursion}), we know that for $t\geq 2$, $w_t\geq \left(1 - \frac{d +\lambda}{n}\right) w_{t-1} - \lambda$. Let us consider the sequence $(a_{t})_{t \ge 1}$, where $a_{1}:=w_1=n$ and $a_{t}:=\left(1 - \frac{d +\lambda}{n}\right) a_{t-1} - \lambda$
for each $t \ge 1$. This sequence has an exact solution, in which 
\[
	a_{t} = n \left(1+ \frac{\lambda}{d+\lambda} \right) \left(1 - \frac{d + \lambda}{n}\right)^{t-1} - \frac{\lambda n}{d + \lambda}
\]
for each $t \ge 1$.
In particular, observe that $a_{2} = n -d -2\lambda$. On the other hand, we know that $v_{1}$
and all of its neighbours are in $B_2$, so we have that $w_{2} = n -d -1$. 
Moreover, the assumption that $3\leq d \le \sqrt{n}$, together with Lemma \ref{lem:eigenvalue_properties} imply that $\lambda \ge \frac{1}{2}$. We may therefore conclude that $w_{2} \ge a_{2}$. It follows that $w_t\geq a_t$ for all $t \ge 1$, as the function $f(x)=\left(1 - \frac{d +\lambda}{n}\right) x - \lambda$ is monotone increasing. We may therefore analyze the sequence $(a_{t})_{t \ge 1}$ to lower bound how long the first phase lasts. In particular, observe that for all $1\le t \le \log \left( \frac{2\lambda}{d+ \lambda}\right)\left(\log \left({1 - \frac{d+ \lambda}{n}}\right) \right)^{-1}$,
\[
	w_{t} \ge a_{t} > \frac{\lambda n}{d+\lambda}.
\]

Then the $Z$-Grundy dominating sequence $S$ returned by the algorithm has length $k \ge t_{1} \ge \log \left( \frac{2\lambda}{d+ \lambda}\right)\left(\log \left({1 - \frac{d+ \lambda}{n}}\right) \right)^{-1}$.
By Theorem \ref{thm:ZGrundy}, we may conclude that
\[
	Z(G) \le n - \log \left( \frac{2\lambda}{d+ \lambda}\right)\left(\log \left({1 - \frac{d+ \lambda}{n}}\right) \right)^{-1}.
\]
Now, since $\frac{d+ \lambda}{n} \rightarrow 0$, we have $1 - \frac{d + \lambda}{n} = (1+o(1))e^{-(d+\lambda)/n}$, and thus 
\[
\left(\log \left({1 - \frac{d+ \lambda}{n}}\right) \right)^{-1}=-\frac{n}{d+\lambda}
(1+o(1))\]
thereby completing the proof.
\end{proof}













To conclude the section, we consider the case when for each $n \ge 1$, we are given a random $d$-regular graph
$\mathcal{G}_{n,d}$. 
The value of $\lambda$ for random $d$-regular graphs has been studied extensively. A major result due to Friedman~\cite{Fri} is the following:
For every fixed  $\epsilon > 0$ and for $G\in \mathcal{G}_{n,d}$, a.a.s.\
$$
\lambda(G) \le 2 \sqrt{d-1}+ \epsilon.
$$

As a result, the above proposition implies that for each fixed $d$ a.a.s,
\[
	Z(\mathcal{G}_{n,d}) \le   n - (1+o(1))\log\left( \frac{d + \lambda}{2 \lambda} \right)\frac{n}{d+\lambda},
\]
where $\lambda=2\sqrt{d}$. If we consider large but constant $d$, then this implies the main result of the section.
\begin{theorem}
For any fixed $\epsilon > 0$ and for sufficiently large fixed $d=d(\epsilon)$, it holds that a.a.s. $\mathcal{G}_{n, d}$ satisfies
\[
	Z(\mathcal{G}_{n,d}) \le   n \left(1  - \frac{(1{-} \epsilon)\log d}{2d} \right).
\]
\end{theorem}

\section{Lower bound and bipartite holes}\label{section lower bound}

 In this section we improve on a comment made in~\cite{Benny}, where they observed that if a graph $G$ does not contain two disjoint sets of $q$ vertices each, with no edges from one set to the other (we call such a substructure a $q$-bipartite hole), then $Z(G) \ge n - 2q$. The reason for this is that if $S=(v_1,\dots,v_{2q})$ is a $Z$-sequence with corresponding witness sequence $W=(w_1,\dots,w_{2q})$, the sets $\{v_1,\dots,v_{q}\}$ and $\{w_{q+1},w_{q+2},\dots,w_{2q}\}$ form a $q$-bipartite hole. In \cite{Benny} they use a known (but loose) result stating that a.a.s. $\mathcal{G}_{n, d}$ has no $\frac{20 n \log d}{d}$-bipartite hole. We improve this and show the following: 
\begin{theorem}
  For any fixed $\epsilon > 0$ and for sufficiently large $d=d(\epsilon)$, it holds that a.a.s.\ $\mathcal{G}_{n, d}$ has no $(1+\epsilon)\frac{2 n \log d}{d}$-bipartite hole. As a result, a.a.s.\ 
  $$
  Z(\mathcal{G}_{n, d}) \ge n \left(1 - (1+\epsilon) \frac{4 \log d}{d} \right).
  $$
\end{theorem}

\begin{proof}
We use the first moment method. Fix some real numbers $a, b$ and let $X_{a, b}$ be the expected number of induced subgraphs consisting of disjoint sets $A_1, A_2$ of $an$ vertices each, such that there are no edges between $A_1$ and $A_2$ and there are exactly $bn$ edges within $A_1$ (the number of edges within $A_2$ is unspecified here). For our first moment calculation, the expected number of bipartite holes with $an$ vertices on each side is $\sum_b X_{a, b}$. Since the sum is over a linear number of terms, it suffices to choose $a$ such that $X_{a, b}$ is exponentially small for all $b$.

Using the pairing model described in Section~\ref{sec-pairing-model}, let us estimate $X_{a, b}$ for given parameters $a$ and $b$. First we choose our sets $A_1$, $A_2$ so we have 
\[
\binom{n}{an \; an\; n-2an} = \frac{n!}{(an)! (an)! (n-2an)!}
\]
choices. Next we will choose $2bn$ configuration points from among the $dan$ points in $A_1$, so we have 
\[
\binom{dan}{2bn} = \frac{(dan)!}{(2bn)!(dan-2bn)!}
\]
choices. Now we put a matching on these $2bn$ points so we have 
\[
(2bn)!! = \frac{(2bn)!}{2^{bn}(bn)!}
\]
choices. Now for the rest of the $dan -2bn$ configuration points in $A_1$, they just need to be matched to any points outside of $A_1 \cup A_2$, so there are 
\[
(dn-2dan)_{(dan-2bn)} = \frac{(dn-2dan)!}{(dn-3dan+2bn)!}
\]
choices. At this point there are $dn-2dan+2bn$ 
many configuration points left unmatched and any of them can match to each other, so we have 
\[
(dn-2dan+2bn)!!= \frac{(dn-2dan+2bn)!}{2^{\frac{dn-2dan+2bn}{2}}\rbrac{\frac{dn-2dan+2bn}{2}}!}
\]
choices. Now we will multiply the previous several lines, and divide by the number of possible matchings, that is, 
\[
(dn)!! = \frac{(dn)!}{2^\frac{dn}{2} \rbrac{\frac{dn}{2}}!}
\]
and after a little cancellation we get
\[
X_{a,b}=\frac{n!(dan)!(dn-2dan)!(dn-2dan+2bn)!\rbrac{\frac{dn}{2}}!}{(an)! (an)! (n-2an)!(dan-2bn)!(bn)!(dn-3dan+2bn)!\rbrac{\frac{dn-2dan+2bn}{2}}! (dn)! }\cdot2^{dan - 2bn}.
\]
Now, by Stirling's formula $n! = (1+o(1)) \sqrt{2\pi n} \rbrac{\frac{n}{e}}^n = \rbrac{\frac{n}{e}}^n \exp(o(n))$, so we estimate 
\[
\frac{\St{n}\St{dan}\St{dn-2dan}\St{dn-2dan+2bn}\St{\frac{dn}{2}}\cdot2^{dan - 2bn +o(n)}}{\left(\frac{an}e\right)^{2an}\hspace{-0.15cm}
\St{n-2an}\hspace{-0.15cm}\St{dan-2bn}\hspace{-0.15cm}\St{bn}\hspace{-0.15cm}\St{dn-3dan+2bn}\hspace{-0.15cm}\St{\frac{dn-2dan+2bn}{2}}\hspace{-0.2cm}\St{dn} }
\]
and now a large power of $\frac{n}{e}$ cancels and we are left with
\[
X_{a,b}=\frac{(da)^{dan}(d-2da)^{dn-2dan}(d-2da+2b)^{dn-2dan+2bn} \rbrac{\frac{d}{2}}^\frac{dn}{2} \cdot2^{dan - 2bn +o(n)}}{a^{2an}(1-2a)^{n-2an}(da-2bn)^{dan-2bn}b^{bn} (d-3da+2b)^{dn-3dan+2bn}\rbrac{\frac{d-2da+2b}{2}}^{\frac{dn-2dan+2bn}{2}}d^{dn}}.
\]
Letting $g(x) := x \log x$, and 
\begin{align}
f(a, b, d):= &g(da)+g(d-2da)+g(d-2da+2b)+g\rbrac{\frac{d}{2}} - 2g(a)-g(1-2a)-g(da-2b)-g(b)\nonumber\\
& \quad -g(d-3da+2b) -g\rbrac{\frac{d-2da+2b}{2}} - g(d) +\rbrac{da - 2b}\log 2 \label{eqn:f}
\end{align}
we can then write 
\[
X_{a, b}= \exp\rbrac{f(a, b, d)n + o(n)}.
\]
So now, to prove that a.a.s.\ the graph has no bipartite hole with $an$ vertices on each side (for some $a$ which may depend on $d$), it suffices to show that   $f(a, b, d)$ is negative for all $b$. Note that 
\begin{align*}
\frac{\partial f}{\partial b} &= 2 \ln(d-2da+2b)+ 2 \ln (da-2b) - \ln b -  2 \ln(d-3da+2b) -\ln\rbrac{\frac{d-2da+2b}{2}} - 2 \ln 2\\
& = \ln \rbrac{\frac{(d-2da+2b)^2(da-2b)^2}{4b(d-3da+2b)^2 \rbrac{\frac{d-2da+2b}{2}}}} = \ln \rbrac{\frac{(d-2da+2b)(da-2b)^2}{2b(d-3da+2b)^2 }}=0
\end{align*}
when 
\[
(d-2da+2b)(da-2b)^2=2b(d-3da+2b)^2.
\]
Expanding, and then factoring again,
\[
d(1-2a)(d-2ad+2b)(d^2a^2+4dab-4b^2-2db)=0.
\]
Solving this for $b$, it is clear that the solution we want is from the quadratic factor, 
\begin{equation}\label{eqn:b}
b = \frac{2da-d + \sqrt{\rbrac{2da-d}^2 +4d^2a^2}}{4},
\end{equation}
in other words we claim that, given fixed $a, d$, the value of $b$ above maximizes $f(a, b, d)$.

Now we imagine $d \rightarrow \infty$, and $a = \frac{K \log d}{d}$. The value of $b$ becomes 
\begin{align*}
b&= \frac{2da-d + \sqrt{d^2 -4d^2a  +8d^2a^2}}{4} = \frac{d}{4} \rbrac{2a-1+\sqrt{1-4a+8a^2}}\\
& = \frac{d}{4} \rbrac{2a-1+1+ \frac{1}{2}(-4a+8a^2) -\frac{1}{8}(-4a+8a^2)^2+ O(a^3)} =  \frac12 da^2 + O(da^3)
\end{align*}
Since $g'(x) = \log x + 1$, $g''(x)=x^{-1}$ and $g'''(x)=-x^{-2}$ by Taylor's theorem we have 
\[
g(x+\Delta x) = g(x) + (\log x + 1) \Delta x +\frac{(\Delta x)^2}{2x} +O\rbrac{ \frac{(\Delta x)^3}{x^2}}
\]for $|\Delta x|< \frac12 x$. If we plug in $a=\frac{K \log d}{d}$ and $b=\frac12 da^2 + O(da^3)$ we get (asymptotics are in $d$ for the following)
\begin{align*}
f(a, b, d)&= g(da)+g(d-2da)+g(d-2da+2b)+g\rbrac{\frac{d}{2}} - 2g(a)-g(1-2a)-g(da-2b)-g(b)\\
& \quad -g(d-3da+2b) -g\rbrac{\frac{d-2da+2b}{2}} - g(d) +\rbrac{da - 2b}\log 2\\
&= g(da)+\GG{d}{-2da}+ \frac{(-2da)^2}{2d} + \GG{d}{+(-2da+2b)} + \frac{(-2da+2b)^2}{2d} \\
& \quad +g\rbrac{\frac{d}{2}} - 2a \log a -\rbrac{-2a }-\sbrac{\GG{da}{-2b})}-g(b)\\
& \quad -\sbrac{\GG{d}{+(-3da+2b)} + \frac{(-3da+2b)^2}{2d}}\\
& \quad-\sbrac{\GG{\frac d2}{+(-da+b)} + \frac{(-da+b)^2}{d}} - g(d) +\rbrac{da - 2b}\log 2 +o\rbrac{\frac{\log^2 d}{d}}\\
&= -da(\log d + 1) + 2b ( \log da + 1)+ (da-b)\rbrac{\log\frac d2 + 1} -\frac{3}{2}da^2+2a\log d - b \log b\\
& \quad +\rbrac{da - 2b}\log 2 +o\rbrac{\frac{\log^2 d}{d}}\\
&= -b(\log d+1) + 2b ( \log da + 1) -\frac{3}{2}da^2+2a\log d - b \log b-b\log 2 +o\rbrac{\frac{\log^2 d}{d}}\\
&= b -\frac{3}{2}da^2+2a\log d +o\rbrac{\frac{\log^2 d}{d}}\\
&= K \rbrac{2- K} \frac{\log^2 d}{d}+o\rbrac{\frac{\log^2 d}{d}}
\end{align*}
which is negative if $K = 2(1+\epsilon)$ and $d$ is sufficiently large. This completes the proof.
\end{proof}

Finally we will justify numerical lower bounds for small $d$ using the same method as above. For fixed $d$, to show that a.a.s. $\mathcal{G}_{n, d}$ has no $an$-bipartite hole it suffices to show that $f(a, b, d)<0$ where $f$ is defined in equation \eqref{eqn:f} and $b$ is taken to be the value in $\eqref{eqn:b}$. It then follows that $Z(\mathcal{G}_{n, d}) \ge (1-2a)n$. The table below provides a summary for $3 \le d \le 14$ and justifies Table \ref{fig:4}. In particular the third column below corresponds to the lower bounds in Table~\ref{fig:4}. 

\begin{table}[htbp]
\centering
 \begin{tabular}{c|c|c}
$d$ & $a$ & $1-2a$ \\
\hline 
3 & 0.46504& 0.06992 \\
4 & 0.42746&  0.14508\\
5 & 0.39432&  0.21136\\
6 & 0.36609& 0.26782\\
7 & 0.34210& 0.31580\\
8 & 0.32156& 0.35688\\
9 & 0.30378& 0.39244\\
10 & 0.28825& 0.42350\\
11 & 0.27454& 0.45092\\
12 & 0.26233& 0.47534\\
13 & 0.25137& 0.49726\\
14 & 0.24147& 0.51706
\end{tabular} 
\caption{Values of $d$ together with values $a$ such that a.a.s. $\mathcal{G}_{n, d}$ has no $(an)$-bipartite hole.}
\end{table}

\section{Concluding Remarks} 
In this paper, we have analyzed the zero forcing number of random $d$-regular graphs, providing an upper bound by considering a random greedy algorithm which creates a zero forcing set and providing a lower bound by considering the structure of bipartite holes. For any $\epsilon > 0$ and sufficiently large $d$, we find that a.a.s., 
$  (\frac{1}{2}-\epsilon)\frac{\log d}{d} \le 1-\frac{Z(\Gnd)}{n} \le (4+\epsilon)\frac{\log d}{d}$.
Of course, a central open problem is to determine the constant in front of $\log d / d$ and it seems reasonable to guess that it ought to be $2+\sqrt{2}$ as is the case for $G(n,p)$. 
Note that if $S=(v_1,\ldots, v_k)$ is a $Z$-sequence with witness sequence $W = (w_1,\ldots, w_k)$, then $v_i\sim w_i$ for all $1\le i\le k$ and $v_i \not\sim w_j$ whenever $i<j$. The authors in \cite{Benny} call such a structure a $k$-witness and are able to use first and  second moment methods to determine asymptotically, the size of the largest $k$-witness which appears in $G(n,p)$.  In the setting of random regular graphs, the calculations become much more complicated as the edges do not appear independently. For example, the first moment argument given in Section \ref{section lower bound} takes several pages, whereas a corresponding argument in $G(n,p)$ takes only a line or two (as seen in \cite{Benny}). We also wonder whether a different algorithm can be used to give a better general upper bound for large $d$.

\section{Acknowledgement} 

The numerical results presented in this paper (in particular, upper bounds in Table~\ref{zf-table} and the figures which appear in the appendix) were obtained using Julia language~\cite{Julia}. We would like to thank Bogumi\l{} Kami\'nski from SGH Warsaw School of Economics for helping us to implement it.

\newpage
\section{Appendix}

\begin{figure}[ht] 
\begin{centering}
\begin{tabular}{cc}
  \includegraphics[width=0.4\textwidth]{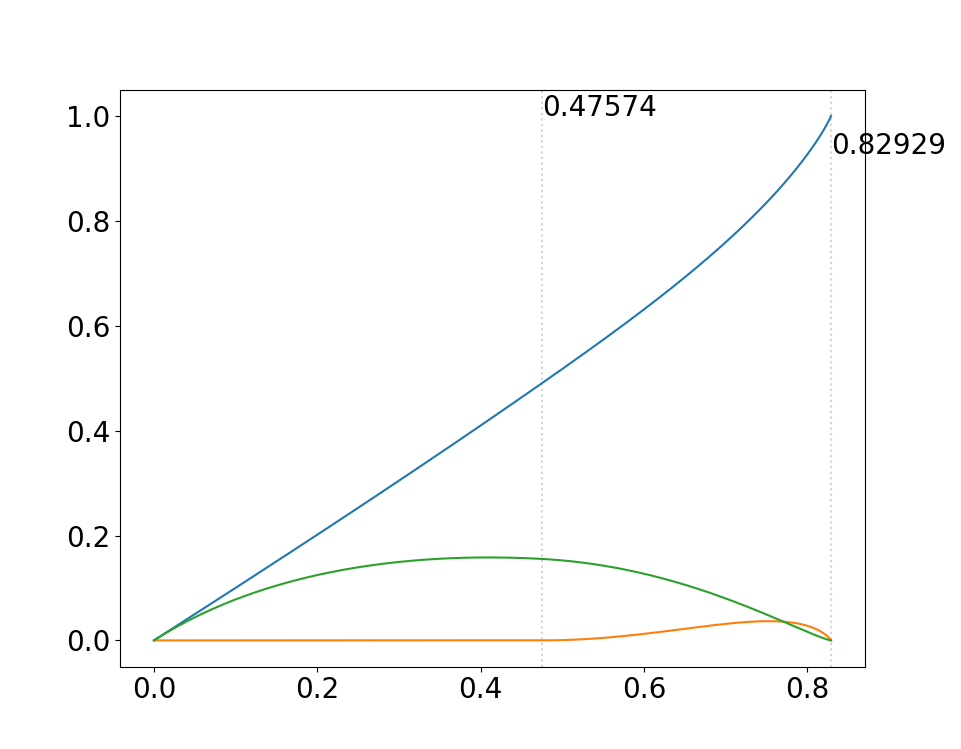} &
  \includegraphics[width=0.4\textwidth]{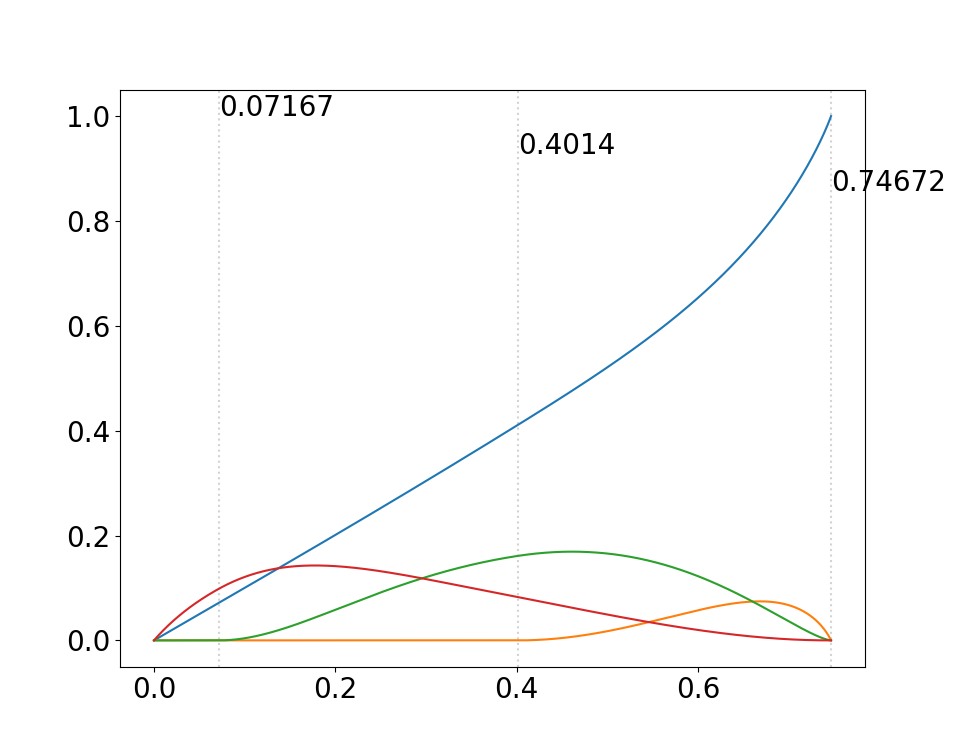} \\
    (a) $d=3$ &
    (b) $d=4$ \\
  \includegraphics[width=0.4\textwidth]{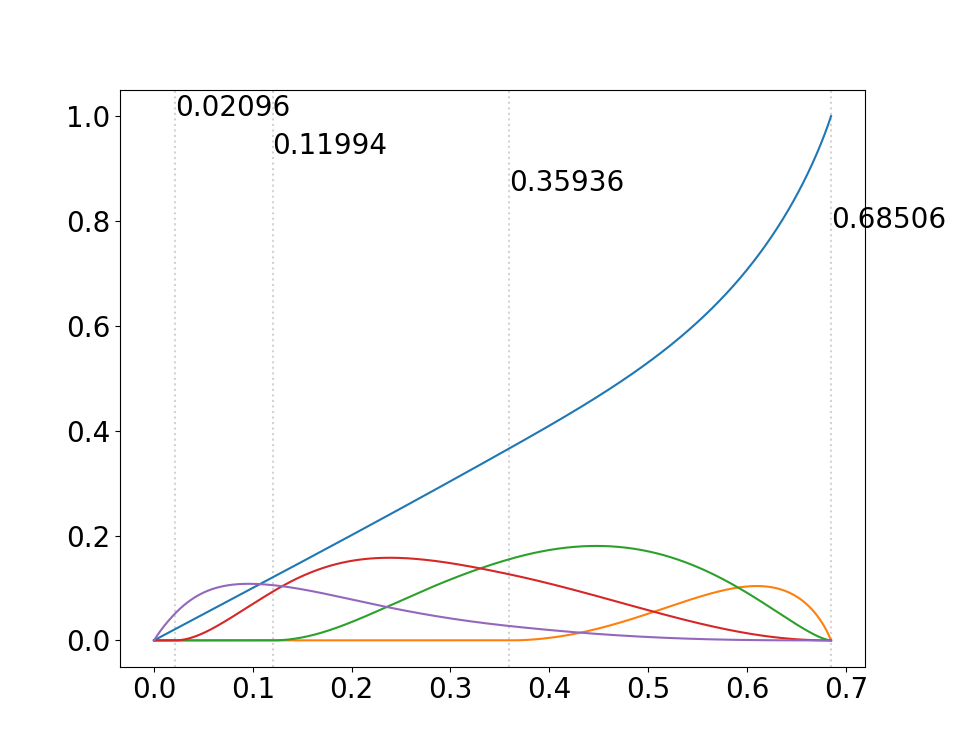} &
  \includegraphics[width=0.4\textwidth]{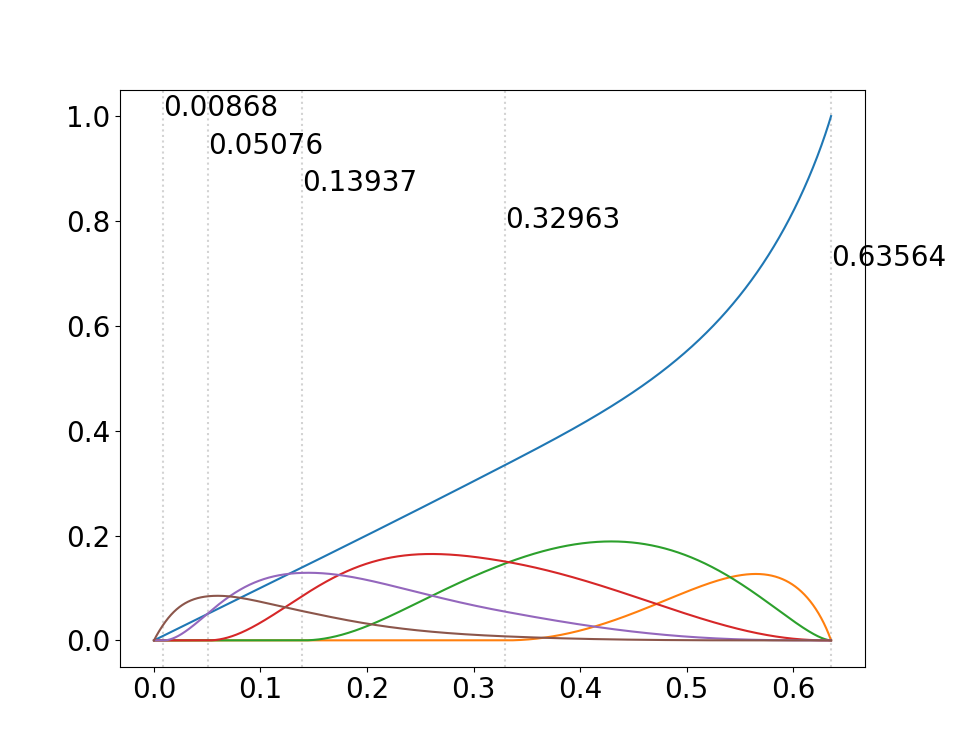} \\
    (c) $d=5$ &
    (d) $d=6$ \\
  \includegraphics[width=0.4\textwidth]{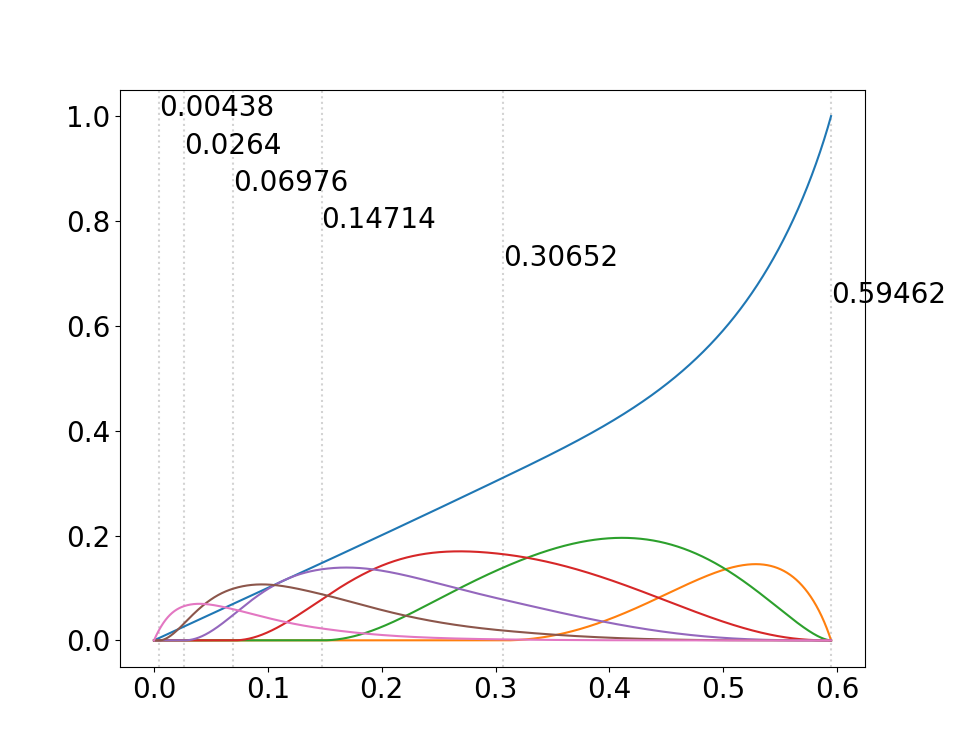} &
  \includegraphics[width=0.4\textwidth]{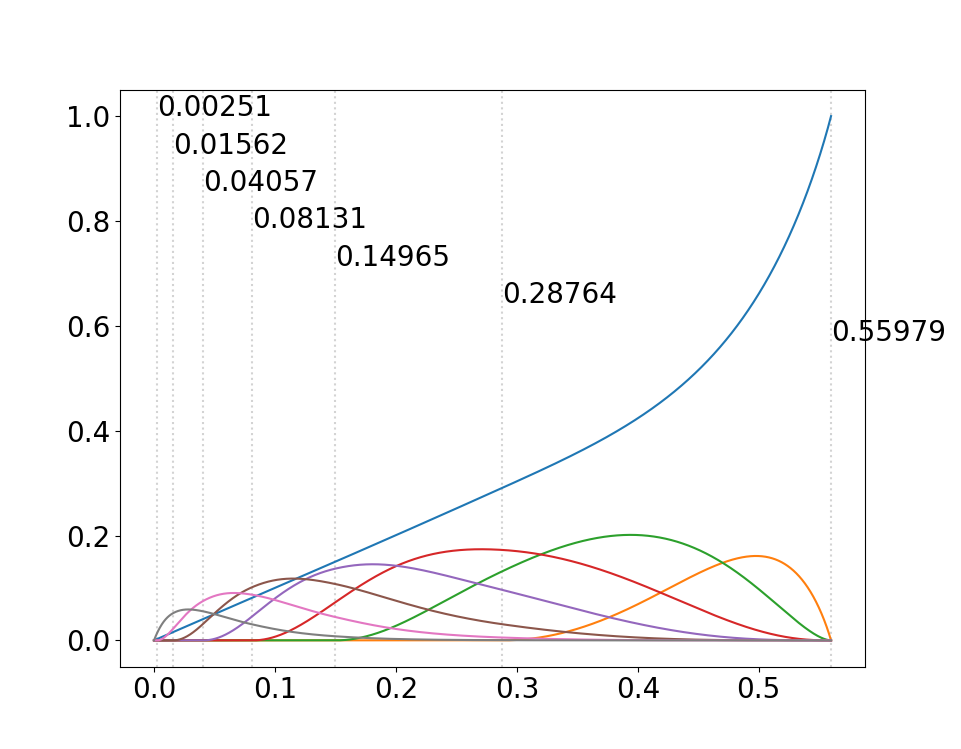} \\
    (e) $d=7$ &
    (f) $d=8$ \\
\end{tabular}\\
\end{centering}
\caption{Solutions to the differential equations for $d=3,4,\ldots 8$.  }
\end{figure}

\begin{figure}[ht] 
\begin{centering}
\begin{tabular}{cc}
  \includegraphics[width=0.4\textwidth]{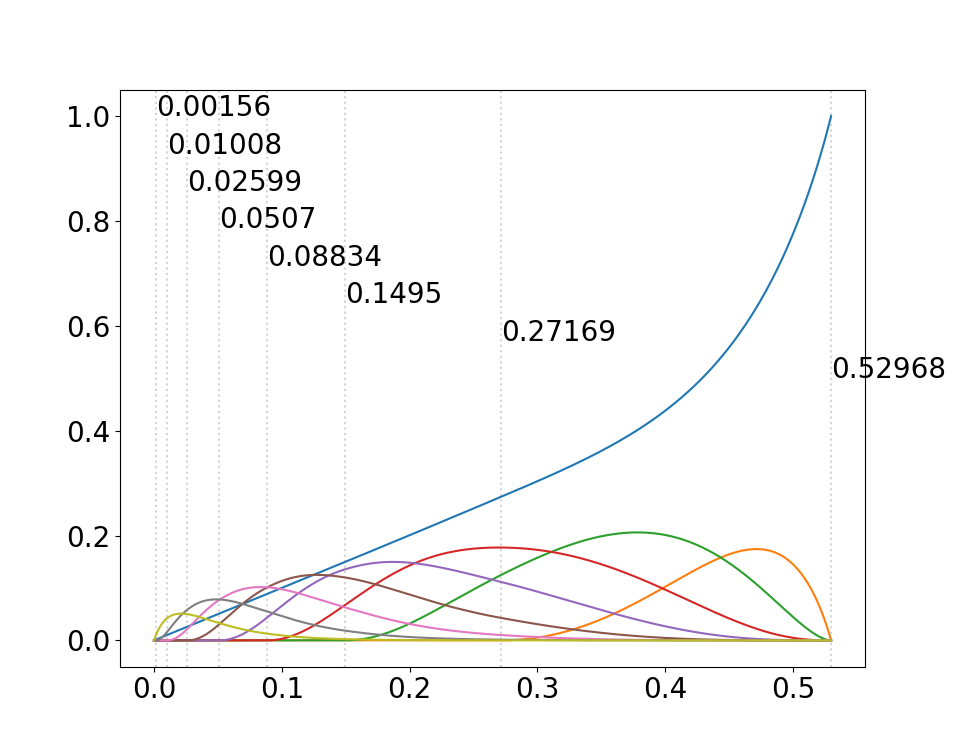} &
  \includegraphics[width=0.4\textwidth]{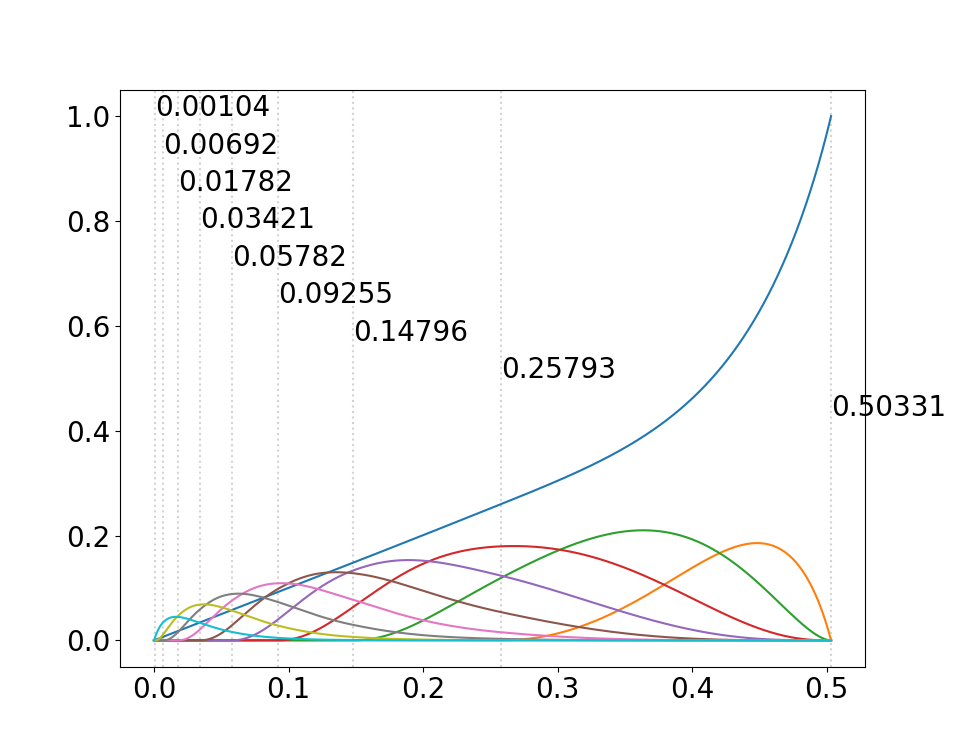} \\
    (g) $d=9$ &
    (h) $d=10$ \\    
  \includegraphics[width=0.4\textwidth]{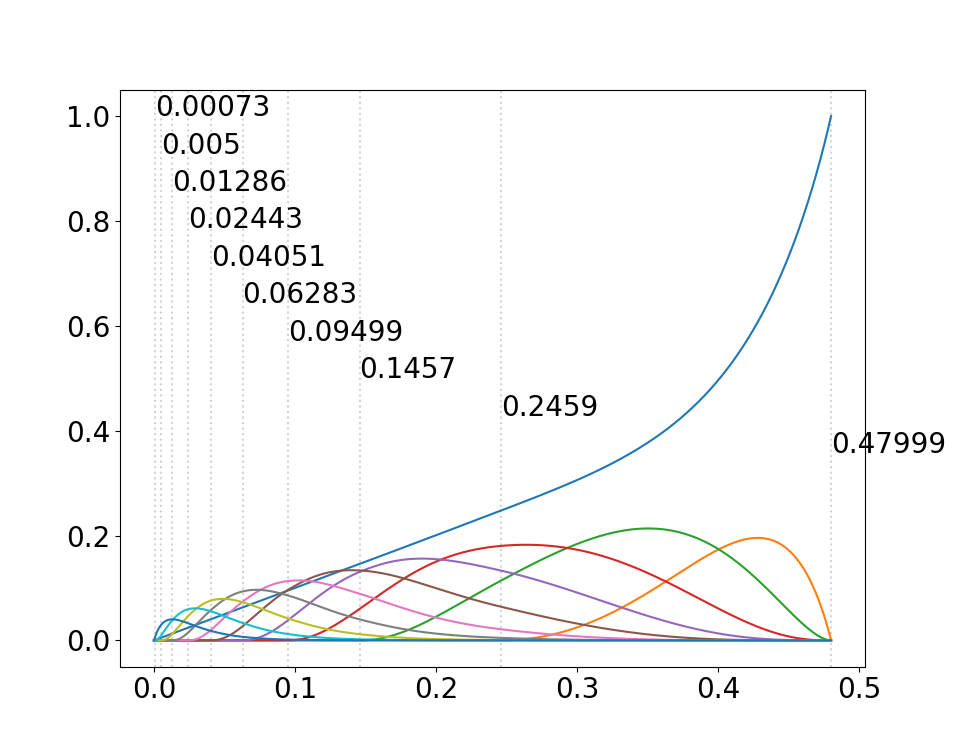} &
  \includegraphics[width=0.4\textwidth]{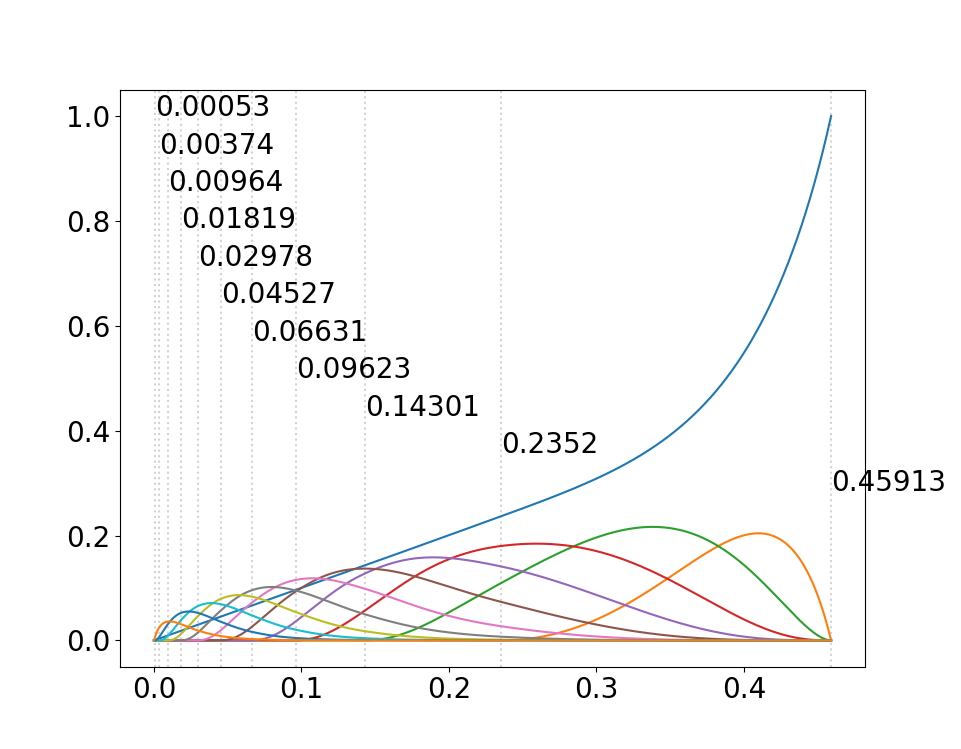} \\
    (i) $d=11$ &
    (j) $d=12$ \\
  \includegraphics[width=0.4\textwidth]{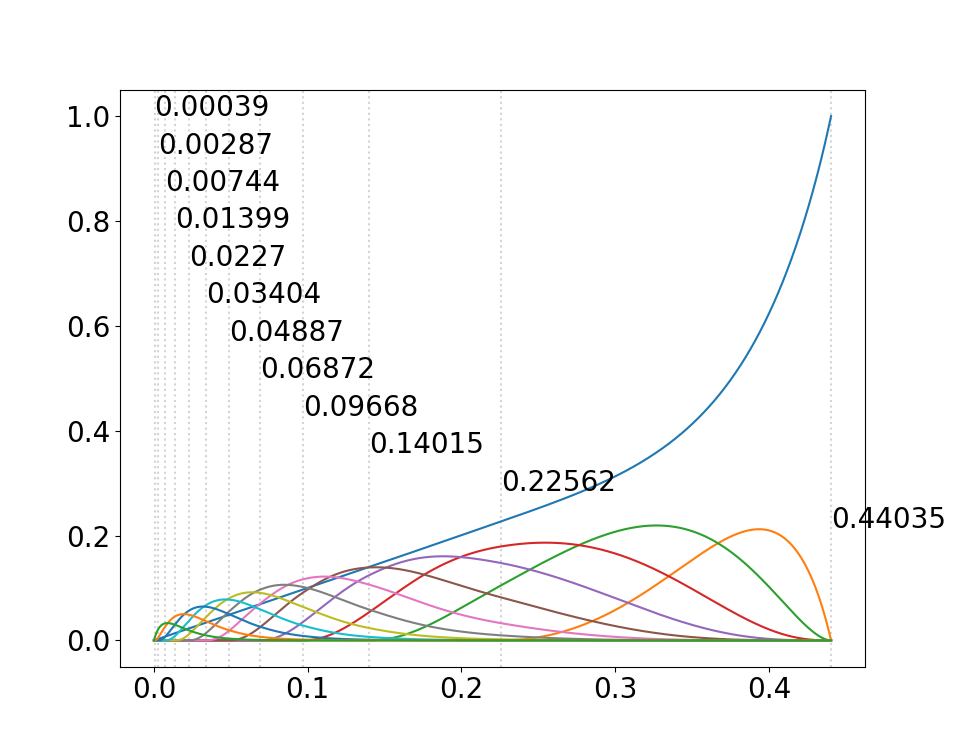} &
  \includegraphics[width=0.4\textwidth]{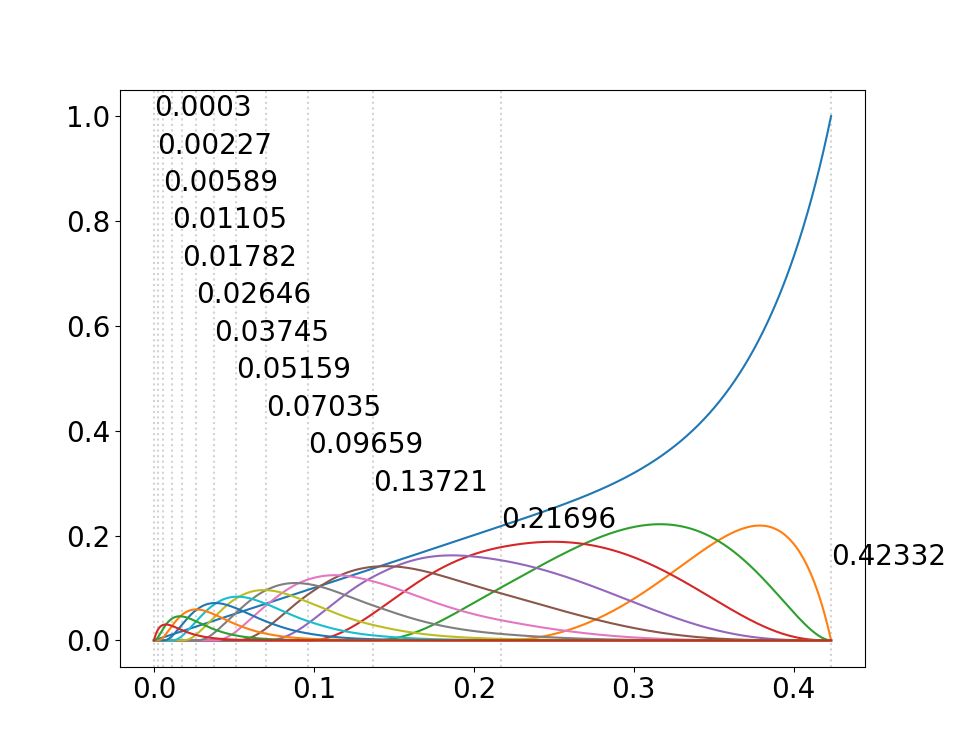} \\
    (k) $d=13$ & 
    (l) $d=14$ \\
\end{tabular}\\
\end{centering}
\caption{Solutions to the differential equations for $d=9,10,\ldots 14$. }
\end{figure}

\end{document}